\documentclass[dvipsnames,a4paper,twoside]{article}
\usepackage{amsthm,amsmath,amssymb,xcolor,soul,xspace,bm}
\usepackage[bookmarks=true,bookmarksopen=true,colorlinks=true,citecolor=blue,linkcolor=blue,urlcolor=blue]{hyperref}

\usepackage[ruled,algo2e]{algorithm2e}

\usepackage[dvipsnames]{xcolor}
\usepackage[ulem=normalem,draft,authormarkup=none,deletedmarkup=xout]{changes}

\newcommand{\tcr}{}

\theoremstyle{plain}
\newtheorem{theorem}{Theorem}[section]
\newtheorem{lemma}[theorem]{Lemma}
\newtheorem{corollary}[theorem]{Corollary}

\theoremstyle{definition}
\newtheorem{definition}[theorem]{Definition}

\newtheorem{assumption}[theorem]{Assumption}

\theoremstyle{remark}
\newtheorem{remark}[theorem]{Remark}

 \numberwithin{equation}{section}
\numberwithin{table}{section}

\title{Boundary bilinear control of semilinear parabolic PDEs: quadratic convergence of the SQP method\thanks{The authors were supported by MICIU/AEI/10.13039/501100011033/ under research project PID2023-147610NB-I00.}}

\author{Eduardo Casas\thanks{Departamento de Matem\'{a}tica Aplicada y Ciencias de la Computaci\'{o}n, E.T.S.I. Industriales y de Telecomunicaci\'on, Universidad de Cantabria, 39005 Santander, Spain
(\texttt{eduardo.casas@unican.es})},
\and
Mariano Mateos\thanks{Departamento de Matem\'{a}ticas, Campus de Gij\'on, Universidad de Oviedo, 33203, Gij\'on, Spain (\texttt{mmateos@uniovi.es})}
}

\ifpdf
\hypersetup{
  pdftitle={Boundary bilinear control of semilinear parabolic PDEs: quadratic convergence of the SQP method},
  pdfauthor={E.~Casas and M.~Mateos}
}
\fi

\newcommand{\dx}{\,\mathrm{d}x}
\newcommand{\dt}{\,\mathrm{d}t}
\newcommand{\Pb}{\mbox{\rm (P)}\xspace}
\newcommand{\uad}{U_{\rm ad}}
\newcommand{\wad}{W_{\rm ad}}
\newcommand{\dimension}{d}
\newcommand{\conormal}{n\!}
\newcommand{\tikhonov}{\kappa}
\newcommand{\mA}{\mathcal{A}}
\newcommand{\mY}{\mathcal{Y}}
\newcommand{\mF}{\mathcal{F}}
\newcommand{\mG}{\mathcal{G}}
\newcommand{\mJ}{\mathcal{J}}
\newcommand{\mL}{\mathcal{L}}
\newcommand{\mT}{\mathcal{T}}
\newcommand{\mW}{\mathcal{W}}
\newcommand{\normalConeW}{\mathcal{S}}
\newcommand{\proj}{\operatorname{Proj}}
\newcommand{\mathe}{\operatorname{e}}
\newcommand{\umin}{{\alpha}}
\newcommand{\umax}{{\beta}}
\newcommand{\vt}{h}
\newcommand{\inc}{n}
\newcommand{\zw}{\zeta}
\newcommand{\etaw}{\xi}
\newcommand{\direction}{\varpi}

\begin{document}

\maketitle

\begin{abstract}
We analyze a bilinear control problem governed by a semilinear parabolic equation. The control variable is the Robin coefficient on the boundary. First-order necessary and second-order sufficient optimality conditions are derived. A sequential quadratic programming algorithm is then proposed to compute local solutions. Starting the iterations from an admissible initial control in an $L^2$-neighborhood of the local solution we prove stability and quadratic convergence of the algorithm in $L^p$ ($p < \infty$) and $L^\infty$ assuming that the local solution satisfies a no-gap second-order sufficient optimality condition and a strict complementarity condition.
\end{abstract}

\begin{quote}
\end{quote}

\begin{quote}
\end{quote}

\section{Introduction}
\label{S1}
In this paper, we analyze a sequential quadratic programming (SQP) method  to solve the following bilinear boundary control problem:
\[
\Pb \min_{u \in \uad} J(u) :=  \int_Q L(x,t,y_u(x,t)) \dx\dt + \int_\Omega l(x,y_u(x,T))\dx + \frac{\tikhonov}{2} \int_\Sigma u^2(x,t) \dx\dt,
\]
where $y_u$ is the state associated with the control $u$, solution of the equation
\begin{equation}
\left\{\begin{array}{l} \displaystyle\frac{\partial y}{\partial t} + Ay + a(x,t,y) = 0\ \  \mbox{in } Q = \Omega \times (0,T),\vspace{2mm}\\  \partial_{\conormal_A} y + uy = g\ \ \mbox{on }\Sigma = \Gamma \times (0,T), \ y(x,0) = y_0(x) \ \ \text{in } \Omega. \end{array}\right.
\label{E1.1}
\end{equation}
Here we assume that $0 < T < \infty$ and $\Omega \subset \mathbb{R}^{\dimension}$, $\dimension = 2$ or $3$, is a bounded open connected set with a Lipschitz boundary $\Gamma$. The precise assumptions of the problem are given in Sections \ref{S2} and \ref{S3}. Throughout this paper, we assume that $\tikhonov >0$ and the set of admissible controls is defined as
\[
\uad = \{ u \in L^2(\Sigma) :  \umin \leq u(x,t) \leq \umax \text{ a.e. in }  \Sigma\}\ \text{ with } 0 \le \umin < \umax < \infty.
\]

In this paper, we analyze the Lagrange-Newton SQP method proposed in \cite{Troltz1999} to solve problem (P). This method treats the three variables --state, adjoint state and control-- as independent variables related by the state and adjoint state equations, which are seen as equality constraints. The reader is referred to \cite{CM2025b} for another version of the SQP method that focuses on the control variable, treating the state and adjoint state as functions of the control. Our aim is to prove quadratic convergence of the algorithm. There are quite a few papers devoted to the analysis of the SQP method applied to the solution of control-constrained optimal control problems. For the convergence analysis of the Lagrange-Newton SQP method for control problems of evolutionary partial differential equations the reader is referred to \cite{Goldberg-Troltz1998}, \cite{Troltz1999}, \cite{Wachsmuth2007} or \cite{Hoppe-Neitzel2021}.

Although our problem fits in the more general framework provided by \cite{Troltz1999}, a careful study of the specific setting investigated in the work at hand leads to the improvement of the existing results in the following sense.
In our main results (see Theorem \ref{VV-T5.10} and Corollary \ref{R5.12}), we prove quadratic convergence in $L^q(\Sigma)$ for all $2(\dimension + 1) \le q\le \infty$ to a local solution of \Pb assuming only no-gap second-order sufficient optimality conditions, a strict complementarity condition, and an admissible initial control in an $L^2(\Sigma)$-neighborhood of the optimal control. In the more general framework of \cite{Troltz1999}, the proof of the quadratic convergence requires the selection of the initial control in an $L^\infty$-neighborhood of the optimal control and this quadratic convergence is proved only in $L^\infty$.
 
 No-gap second order conditions obtained in Theorem \ref{T3.11} and strict complementarity conditions introduced in Definition \ref{D3.12} are the natural extensions to infinite dimension of the conditions imposed in finite dimensional optimization to obtain quadratic convergence  results for the SQP method for constrained optimization problems. In contrast, in the afore-mentioned papers, coercivity of the second derivative of the Lagrangian on a subspace $E^\tau_{\bar u}$ that takes into account the so-called strongly active constraints must be assumed. We are able to prove in Theorem \ref{T3.13} that this strong condition can be deduced from our more natural assumptions.

 Regarding the possibility of choosing the initial control in an $L^2$-neighborhood of the solution, to our best knowledge, this has been proved only in \cite[Theorem 6.10]{Hoppe-Neitzel2021} for controls only depending on time. The approach followed in \cite{Hoppe-Neitzel2021} cannot be applied to our case due to the bilinear structure of our control problem. To overcome this difficulty, instead of following the usual technique of proof employed in the previous papers, which relies on existing abstract results for the generalized Newton method, we give a complete and detailed proof of the quadratic convergence.

The fact of having a boundary bilinear control and a nonlinear parabolic equation introduces several technical difficulties, among which we may highlight the lack of differentiability in $L^2(\Sigma)$ of the control-to-state mapping. As far as we know, the second order analysis for this problem has not been addressed in the literature. The reader is referred to \cite{CCM2023} and \cite{CCM-NM2024} and the references therein for the analysis of stationary bilinear control problems.

The plan of the paper is as follows. In Section \ref{S2} we analyze the state equation and prove the differentiability properties of the control-to-state mapping. Section \ref{S3} is devoted to the analysis of the control problem. The differentiability properties of the objective functional are studied. We also prove the differentiability of the control-to-adjoint state mapping and express the second derivative of the objective functional in terms of the so-called second adjoint state. First- and second-order optimality conditions are also studied in this section.
In Section \ref{S4}, we reformulate the optimality conditions in terms of the Lagrangian function of the control problem. In Section \ref{S5} we analyze the SQP method. Finally, in Section \ref{S6} numerical experiments are presented. Some technical proofs are postponed to an appendix.

\section{Analysis of the state equation}
\label{S2}
\setcounter{equation}{0}

In this section we address the question of solvability of the state equation and the differentiability of the control-to-state mapping. The following hypotheses are assumed along this paper.

\begin{assumption}\label{A2.1}
The operator $A$ is defined in $\Omega$ by the expression
\[
Ay = - \sum_{i,j=1}^{\dimension} \partial_{x_j} [ a_{ij}(x) \partial_{x_i} y] \]
with $a_{ij} \in L^{\infty}(\Omega)$ for $1\leq i,j \leq \dimension$ satisfying for some $\lambda_A,\Lambda_A >0$
\begin{equation}
\lambda_A \vert h\vert ^2 \geq \sum_{i,j=1}^{\dimension} a_{ij}(x) h_i h_j \geq \Lambda_A \vert h\vert ^2 \quad \text{for a.a.} \ x \in \Omega \ \text{and} \ \forall h \in \mathbb R^\dimension.
\label{E2.1}
\end{equation}
The normal derivative $\partial_{\conormal_A} y$ is formally defined by
\[
\partial_{\conormal_A} y = \sum_{i,j=1}^n a_{ij}(x) \partial_{x_i} y(x) \conormal_j(x),
\]
where $\conormal(x)$ denotes the outward unit normal vector to $\Gamma$ at the point $x$. Due to the Lipschitz regularity of $\Gamma$ such a vector $\conormal(x)$ exists for almost all $x \in \Gamma$. The reader is referred to \cite[page 511]{Dautray-Lions2000} for a rigorous definition of the normal derivative in a trace sense with $\partial_{\conormal_A} y(\cdot,t) \in H^{-\frac{1}{2}}(\Gamma)$ for almost all $t \in (0,T)$.
\end{assumption}
\begin{assumption} \label{A2.2}
We assume that $a:Q \times \mathbb{R} \longrightarrow \mathbb{R}$ is a Carath\'eodory function of class $C^2$ with respect to the last variable satisfying the following properties for a.a. $(x,t) \in Q$:
\begin{align*}
&\bullet a(\cdot,\cdot,0) \in L^r(0,T;L^s(\Omega)) \text{ for some } r,s \ge 2 \text{ with } \frac{1}{r} + \frac{\dimension}{2s} < 1,  \\
&\bullet \exists C_a \in \mathbb{R} \text{ such that } \frac{\partial a}{\partial y}(x,t,y) \geq C_a\ \ \forall y\in\mathbb{R}, \\
&\bullet \forall M > 0 \ \exists C_{a,M} \text{ such that }
 \sum_{j=1}^2  \Big\vert  \frac{\partial^j a}{\partial y^j} (x,t,y) \Big\vert  \leq C_{a,M}  \  \forall \vert y\vert \leq M, \\
&\bullet \forall M > 0 \ \exists K_{a,M} \text{ such that } \ \Big\vert     \frac{\partial^2 a}{\partial y^2}  (x,t,y_1) -  \frac{\partial^2 a}{\partial y^2}  (x,t,y_2) \Big\vert  \leq K_{a,M}\vert y_1 - y_2\vert ,\\
& \hspace{2.65cm}\text{ for all } \vert y_1\vert, \vert y_2\vert \leq M,
\end{align*}
where the constants $C_{a}$, $C_{a,M}$ and $K_{a,M}$ are independent of $(x,t)$.
\end{assumption}

\begin{assumption} \label{A2.3}
The initial condition $y_0$ belongs to $L^\infty(\Omega)$ and the boundary data satisfies  $g \in L^2(\Sigma) \cap L^{\hat r}(0,T;L^{\hat s}(\Gamma))$ with $\frac{1}{\hat r} + \frac{\dimension-1}{2\hat s} < \frac{1}{2}$.
\end{assumption}

In the sequel we use the following standard notation
\[
W(0,T) = \{y \in L^2(0,T;H^1(\Omega)) : \frac{\partial y}{\partial t} \in L^2(0,T;H^1(\Omega)^*)\}.
\]
We know that this is a Hilbert space when it is endowed with the graph norm. We also consider the Banach space $Y = W(0,T) \cap L^\infty(Q)$ endowed with the norm $\Vert y\Vert _Y = \Vert y\Vert _{W(0,T)} + \Vert y\Vert _{L^\infty(Q)}$. In addition, we denote the norm of the space $L^\infty(0,T;L^2(\Omega))\cap L^2(0,T;H^1(\Omega))$ by
\begin{equation*}
\Vert y\Vert _Q = \Vert y\Vert _{L^\infty(0,T;L^2(\Omega))} + \Vert y\Vert _{L^2(0,T;H^1(\Omega))}.
\end{equation*}.

\begin{theorem}\label{T2.6}
For every $u \in L^2(\Sigma)$ such that $u(x,t) \ge 0$ for almost all $(x,t) \in \Sigma$, the state equation \eqref{E1.1} has a unique solution $y_u \in Y$ and the following estimates hold
\begin{align}
&\Vert y_u\Vert _Q \le K_2\big(\Vert y_0\Vert _{L^2(\Omega)} + \Vert a(\cdot,0)\Vert _{L^2(Q)} + \Vert g\Vert _{L^2(\Sigma)}\Big),\label{E2.2}\\
&\Vert y_u\Vert _{L^\infty(Q)} \le K_\infty\Big(\Vert y_0\Vert _{L^\infty(\Omega)} + \Vert a(\cdot,0)\Vert _{L^r(0,T;L^s(\Omega))} + \Vert g\Vert _{L^{\hat r}(0,T;L^{\hat s}(\Gamma))}\Big),\label{E2.3}\\
&\Vert y_u\Vert _{W(0,T)} \le K_W\big(\Vert a(\cdot,0)\Vert _{L^2(Q)} +[C_{a,M_u}\! +\! \Vert u\Vert _{L^2(\Sigma)}]M_u + \Vert g\Vert _{L^2(\Sigma)}\Big),\label{E2.4}
\end{align}
where  $M_u = \Vert y_u\Vert _{L^\infty(Q)}$,  $C_{a,M_u}$ is defined according to Assumption \ref{A2.2}, and the constants $K_2$, $K_\infty$, and $K_W$ are independent of $u$. Finally,  if $\{u_k\}_{k = 1}^\infty$ is a sequence of nonnegative functions such that $u_k \rightharpoonup u$ in $L^p(\Sigma)$ with $p > \dimension + 1$, then the following convergences hold
\begin{equation}
y_{u_k} \rightharpoonup y_u \text{ in } W(0,T)\ \ \text{ and }\ \ \lim_{k \to \infty}\Vert y_{u_k} - y_u\Vert _{C(\bar Q)} = 0.
\label{E2.5}
\end{equation}
\end{theorem}

\begin{proof}
Performing the change of variable $z(x,t) = \mathe^{-2|C_a|t}y(x,t)$ the equation \eqref{E1.1} is transformed in
\begin{equation}
\left\{\begin{array}{l} \displaystyle\frac{\partial z}{\partial t} + Az + 2|C_a|z + \tilde a(x,t,z) =  -\mathe^{-2|C_a|t}a(x,t,0)\ \  \mbox{in } Q ,\vspace{2mm}\\  \partial_{\conormal_A}z  + uz = \mathe^{-2|C_a|t}g\ \ \mbox{on }\Sigma , \ z(x,0) = y_0(x) \ \ \text{in } \Omega, \end{array}\right.
\label{E2.6}
\end{equation}
where $\tilde a(x,t,s) = \mathe^{-2|C_a|t}[a(x,t,\mathe^{2|C_a|t}s) - a(x,t,0)]$. From Assumption \ref{A2.2} we infer
\begin{equation}
\tilde a(x,t,0) = 0 \ \text{ and } \ \frac{\partial}{\partial z}\Big(2|C_a|z + \tilde a(x,t,z)\Big) = 2|C_a| + \frac{\partial a}{\partial y}(x,t,\mathe^{2|C_a|t}z) \ge |C_a|.
\label{E2.7}
\end{equation}
For every integer $k \ge \Vert y_0\Vert _{L^\infty(\Omega)}$ we define $\tilde a_k(x,t,s) = \tilde a(x,t,\proj_{[-k,+k]}(s))$. Now, we consider the equation
\[
\left\{\begin{array}{l} \displaystyle\frac{\partial z_k}{\partial t} + Az_k + 2|C_a|z_k + \tilde a_k(x,t,z_k) = -\mathe^{-2|C_a|t}a(x,t,0)\ \  \mbox{in } Q ,\vspace{2mm}\\  \partial_{\conormal_A} z_k + u\proj_{[-k,+k]}(z_k) = \mathe^{-2|C_a|t}g\ \ \mbox{on }\Sigma , \ z_k(x,0) = y_0(x) \ \ \text{in } \Omega. \end{array}\right.
\]
An easy application of Schauder's fixed point theorem along with \eqref{E2.7}  leads to the existence of a unique solution $z_k \in W(0,T)$ of the above equation. Moreover, setting $z_{k,j}(x,t) = z_k(x,t) - \proj_{[-j,+j]}(z_k(x,t))$ for every integer $j \ge \Vert y_0\Vert _{L^\infty(\Omega)}$, noting that
\[
[2\vert C_a\vert z_k + \tilde a_k(x,t,z_k)]z_{k,j} \ge  \vert C_a\vert z_{k,j}^2\ \text{ and }\ u\proj_{[-k,+k]}(z_k) z_{k,j} \ge 0,
\]
and applying \cite{Casas-Kunisch2025} we deduce the existence of a constant $M$ such that
\[
\Vert z_k\Vert _{L^\infty(Q)} \le M = C\Big(\Vert y_0\Vert _{L^\infty(\Omega)} + \Vert a(\cdot,0)\Vert _{L^r(0,T;L^s(\Omega))} + \Vert g\Vert _{L^{\hat r}(0,T;L^{\hat s}(\Gamma)}\Big)\ \ \forall k.
\]
Then, taking $k \ge M$ we get that $\tilde a_k(x,t,z_k) = \tilde a(x,t,z_k)$ and, consequently, $z = z_k$ for all $k \ge M$ is the unique solution of \eqref{E2.6}. Setting $y = \mathe^{2|C_at|}z(x,t)$ we deduce that $y$ is the unique solution in $Y$ of \eqref{E1.1} and the estimates \eqref{E2.2}-\eqref{E2.4} follow easily.

Let us prove the last statement of the theorem. Given the sequence of nonnegative functions $\{u_k\}_{k = 1}^\infty \subset L^p(\Sigma)$ converging weakly to $u$ in $L^p(\Sigma)$, we deduce from the first part of the theorem the existence and uniqueness of solutions $\{y_{u_k}\}_{k = 1}^\infty \subset Y$ and $y_u \in Y$. Moreover, $\{y_{u_k}\}_{k = 1}^\infty$ is a bounded sequence in $Y$. Hence, there exists a subsequence, denoted in the same way such that $y_{u_k} \rightharpoonup y$ in $W(0,T)$ and $y_{u_k} \stackrel{*}{\rightharpoonup} y$ in $L^\infty(Q)$. Setting $z_k = y_{u_k} - y_u$ and subtracting the equations satisfied by $y_{u_k}$ and $y_u$ and applying the mean value theorem we obtain
\[
 \left\{\begin{array}{l} \displaystyle\frac{\partial z_k}{\partial t} + Az_k + \frac{\partial a}{\partial y}(x,t,\hat y_k)z_k = 0\ \  \mbox{in } Q,\vspace{2mm}\\  \partial_{\conormal_A} z_k + \tcr{u} z_k = (\tcr{u} - u_k)y_{u_k}\ \ \mbox{on }\Sigma,\ z_k(0) = 0 \ \text{ in } \Omega, \end{array}\right.
\]
where $\hat y_k(x,t) = y_u(x,t) + \theta_k(x,t)(y_{u_k}(x,t) - y_u(x,t))$ with $0 \le \theta(x,t) \le 1$.  Since $\{(\bar u - u_k)y_{u_k}\}_{k = 1}^\infty$ is bounded in $L^p(\Sigma)$ we infer the existence of $\sigma \in (0,1)$ such that $\{z_k\}_{k = 1}^\infty$ is bounded in the space of H\"older functions $C^{0,\sigma}(\bar Q)$. This follows from \cite[Theorem 4.5]{DER2017} taking $s = p$ and $q = \frac{\dimension + 1}{\dimension - 1}$. Since $z_k \rightharpoonup z = y - y_u$ in $W(0,T)$ and taking into account the compactness of the embedding $C^{0,\sigma}(\bar Q) \subset C(\bar Q)$ we get that $z_k \to z$ in $C(\bar Q)$. This implies that $\Vert y_{u_k} - y\Vert _{C(\bar Q)} = \Vert z_k - z\Vert _{C(\bar Q)} \to 0$. Then, it is straightforward to pass to the limit in the state equation satisfied by $y_{u_k}$ and to deduce that $y = y_u$, which concludes the proof.
\end{proof}

In this section and in section \ref{S3}, for the analysis of the state equation, the adjoint state equation, and the control problem we assume that $p > \dimension + 1$ and define the set
\[
\mA_0 = \{u \in L^p(\Sigma) : u(x,t) \ge 0 \text{ for a.a.} (x,t) \in \Sigma\}.
\]
To obtain the results for the quadratic convergence of the SQP method, we will assume $p\geq 2(\dimension+1)$ in sections \ref{S4} and \ref{S5}. Now we address the differentiability of the control-to-state relation.

\begin{theorem}\label{T2.8}
There exists an open subset $\mA$ of $L^p(\Sigma)$ such that $\mA_0 \subset \mA$ and $\forall u \in \mA$ the state equation \eqref{E1.1} has a unique solution $y_u \in Y$. Furthermore, the mapping $G : \mA \longrightarrow Y$ defined by $G(u):=y_u$ is of class $C^2$. For all $u \in \mA$ and all $v, v_1, v_2 \in L^p(\Sigma)$ the functions
$z_{u,v} =G'(u)v$ and $z_{u,(v_1,v_2)}=G''(u)(v_1,v_2)$ are the unique solutions of the equations:
\begin{align}
& \left\{\begin{array}{l} \displaystyle\frac{\partial z}{\partial t} + Az +  \frac{\partial a}{\partial y}(x,t,y_u)z = 0\ \  \mbox{in } Q,\vspace{2mm}\\  \partial_{\conormal_A} z + uz + y_uv = 0\ \ \mbox{on }\Sigma,\ z(0) = 0 \ \text{ in } \Omega, \end{array}\right.
\label{E2.8} \\
& \left\{\begin{array}{l} \displaystyle\frac{\partial z}{\partial t} + Az +  \frac{\partial a}{\partial y}(x,t,y_u)z + \displaystyle \frac{\partial^2 a}{\partial y^2}(x,t,y_u)z_{u,v_1}z_{u,v_2} = 0\ \  \mbox{in } Q,\vspace{2mm}\\  \partial_{\conormal_A}z +  uz + v_1z_{u,v_2} + v_2z_{u,v_1}= 0\ \ \mbox{on }\Sigma, \ w(0) = 0 \ \text{ in } \Omega,\end{array}\right.
\label{E2.9}
\end{align}
where $z_{u,v_i} = G'(u)v_i$, $i=1,2$. Moreover, for every $\bar u \in \mA$ there exist $r > 0$ and $L_{\bar u}$ such that
\begin{equation}
\Vert G''(u_2) - G''(u_1)\Vert  \le L_{\bar u}\Vert u_2 - u_1\Vert _{L^p(\Sigma)}\quad \forall u_1, u_2 \in B_r(\bar u),
\label{E2.10}
\end{equation}
where $B_r(\bar u)$ is the ball in $L^p(\Sigma)$ and $\Vert G''(u_2) - G''(u_1)\Vert $ denotes the norm in the space of bilinear and continuous mappings $\mathcal{B}(L^p(\Sigma) \times L^p(\Sigma),Y)$.
\end{theorem}

\begin{proof}
We define the space
\[
\mY_{A} = \{y \in Y : \frac{\partial y}{\partial t} + Ay \in L^r(0,T;L^s(\Omega))\text{ and } \partial_{\conormal_A}y \in L^p(\Sigma) + L^{\hat r}(0,T;L^{\hat s}(\Gamma))\},
\]
which is a Banach space when it is endowed with the graph norm. We also define the mapping
\begin{align*}
&\mF : \mY_{A} \times L^p(\Sigma) \longrightarrow L^r(0,T;L^s(\Omega)) \times [L^p(\Sigma) + L^{\hat r}(0,T;L^{\hat s}(\Gamma))] \times L^\infty(\Omega)\\
&\mF(y,u) = \Big(\frac{\partial y}{\partial t} + Ay + a(x,t,y),\partial_{\conormal_A}y + uy - g,y(0) - y_0\Big).
\end{align*}
As a straightforward consequence of Assumption \ref{A2.2} we deduce that $\mF$ is of class $C^2$. Moreover, we have that $\mF(y_u,u) = (0,0,0)$ for every $u \in \mA_0$ and the derivative
\begin{align*}
&\frac{\partial\mF}{\partial y}(y_u,u) : \mY_{A} \longrightarrow L^r(0,T;L^s(\Omega)) \times [L^p(\Sigma) + L^{\hat r}(0,T;L^{\hat s}(\Gamma))] \times L^\infty(\Omega)\\
&\frac{\partial\mF}{\partial y}(y_u,u)z = \Big(\frac{\partial z}{\partial t} + Az + \frac{\partial a}{\partial y}(x,t,y_u)z,\partial_{\conormal_A}z + uz,z(0)\Big)
\end{align*}
is linear and continuous. As a consequence of Theorem \ref{T2.6} we deduce that for every $(f,h,z_0) \in L^r(0,T;L^s(\Omega)) \times [L^p(\Sigma) + L^{\hat r}(0,T;L^{\hat s}(\Gamma))] \times L^\infty(\Omega)$ the equation
\[
\left\{\begin{array}{l} \displaystyle\frac{\partial z}{\partial t} + Az + \displaystyle \frac{\partial a}{\partial y}(x,t,y_u)z = f\ \  \mbox{in } Q,\vspace{2mm}\\  \partial_{\conormal_A} z + uz = h\ \ \mbox{on }\Sigma,\ z(0) = z_0 \ \text{ in } \Omega, \end{array}\right.
\]
has a unique solution $z \in \mY_{A}$. Thus, $\frac{\partial\mF}{\partial y}(y_u,u)$ is an isomorphism for every $u \in \mA_0$. From the implicit function theorem we infer that for every $\bar u \in \mA_0$ with associated state $\bar y$ there exist $\varepsilon_{\bar u} > 0$ and $\varepsilon_{\bar y} > 0$ such that for every $u \in B_{\varepsilon_{\bar u}}(\bar u) \subset L^p(\Sigma)$ the equation \eqref{E1.1} has a unique solution $y_u$ in the ball $B_{\varepsilon_{\bar y}}(\bar y) \subset \mY_{A}$. Moreover, the mapping $u \in  B_{\varepsilon_{\bar u}}(\bar u) \to y_u \in B_{\varepsilon_{\bar y}}(\bar y)$ is of class $C^2$. Moreover, we can take $\varepsilon_{\bar u}$ sufficiently small such that for every $u \in B_{\varepsilon_{\bar u}}(\bar u)$ the equation ${\mathcal F}(u,y) = 0$ has unique solution in the space $\mY_{A}$. Indeed, let $y_1,y_2$ denote two solutions of
${\mathcal F}(u,y) = 0$. We set $y=y_2-y_1$, subtract the corresponding equations, and apply the mean value theorem to deduce that $y$ satisfies
\[
\left\{\begin{array}{l} \displaystyle\frac{\partial y}{\partial t} + Ay + \displaystyle \frac{\partial a}{\partial y}(x,t,y_1 + \theta y)y = 0\ \  \mbox{in } Q,\vspace{2mm}\\  \partial_{\conormal_A} y + uy = 0\ \ \mbox{on }\Sigma,\ y(0) = 0 \ \text{ in } \Omega, \end{array}\right.
\]
where $\theta:Q \longrightarrow [0,1]$ is a measurable function. We rewrite this equation as follows
\[
\left\{\begin{array}{l} \displaystyle\frac{\partial y}{\partial t} + Ay + \displaystyle \frac{\partial a}{\partial y}(x,t,y_1 + \theta y)y = 0\ \  \mbox{in } Q,\vspace{2mm}\\  \partial_{\conormal_A} y + \bar uy = (\bar u - u)y\ \ \mbox{on }\Sigma,\ y(0) = 0 \ \text{ in } \Omega. \end{array}\right.
\]
Applying \eqref{E2.3} to the above equation for the case where $\hat r = \hat s = p > \dimension + 1$ we infer
\[
\Vert y\Vert _{L^\infty(Q)} \le C_p\Vert u - \bar u\Vert _{L^p(\Sigma)}\Vert y\Vert _{L^\infty(Q)}.
\]
Hence, taking $\varepsilon_{\bar u} < 1/C_p$ we infer from the above inequality that $y = 0$. Now, defining $\mA = \cup_{\bar u \in \mA_0} B_{\varepsilon_{\bar u}} (\bar u)$ with $\varepsilon_{\bar u} < 1/C_p$ for every $\bar u \in \mA_0$ and $G: \mA \longrightarrow Y$ such that $G(u)=y_u$, we have that $\mA$ is an open subset of $L^p(\Sigma)$ and $G$ is a well defined mapping of class of $C^2$. Equations \eqref{E2.8} and \eqref{E2.9} are obtained differentiating the identity ${\mathcal F}(u,G(u)) = 0$ with respect to $u$.

Let us prove \eqref{E2.10}. Since $G:\mA \subset L^p(\Sigma) \longrightarrow Y$ is of class $C^2$ and using the estimates \eqref{E2.2}-\eqref{E2.4},  we infer that for every $\bar u \in \mA$ there exists a ball $B_r(\bar u) \subset \mA$ and constants $M_{\bar u}$ and $C_{\bar u}$ such that $\forall u \in B_r(\bar u)$ and $\forall v, v_1, v_2  \in L^p(\Sigma)$ the following inequalities hold
\begin{align}
&\Vert y_u\Vert _Y\! \le\! M_{\bar u}, \, \Vert z_{u,v}\Vert _Y\! \le\! M_{\bar u}\Vert v\Vert _{L^p(\Sigma)},\, \Vert z_{u,(v_1,v_2)}\Vert _Y\! \le\! M_{\bar u}\Vert v_1\Vert _{L^p(\Sigma)}\Vert v_2\Vert _{L^p(\Sigma)},\label{E2.11}\\
&\Vert G'(u)\Vert _{\mathcal{L}(L^p(\Sigma),Y)} + \Vert G''(u)\Vert _{\mathcal{B}(L^p(\Sigma \times L^p(\Sigma),Y)} \le C_{\bar u},
\label{E2.12}
\end{align}
where $y_u = G(u)$, $z_{u,v} = G'(u)v$, and $z_{u,(v_1,v_2)} = G''(u)(v_1,v_2)$. Moreover, using the generalized mean value theorem we infer that
\begin{align}
&\Vert y_{u_2} - y_{u_1}\Vert _Y \le C_{\bar u}\Vert u_2 - u_1\Vert _{L^p(\Sigma)} \quad \forall u_1, u_2 \in B_r(\bar u),\label{E2.13}\\
&\Vert z_{u_2,v} - z_{u_1,v}\Vert _Y \le C_{\bar u}\Vert u_2 - u_1\Vert _{L^p(\Sigma)}\Vert v\Vert _{L^p(\Sigma)}\quad \forall v \in L^p(\Sigma),\label{E2.14}
\end{align}
where $y_{u_i} = G(u_i)$ and $z_{u_i,v} = G'(u_i)v$.

For $u_1, u_2 \in B_r(\bar u)$ and $v_1, v_2 \in L^p(\Sigma)$ we denote $z_i = G''(u_i)(v_1,v_2)$, $i = 1, 2$, and $z = z_2 - z_1$. Subtracting the equations satisfied by $z_2$ and $z_1$ we get
\[
 \left\{\begin{array}{l} \displaystyle\frac{\partial z}{\partial t} + Az + \displaystyle \frac{\partial a}{\partial y}(x,t,y_{u_2})z = \Big[\frac{\partial a}{\partial y}(x,t,y_{u_1}) - \frac{\partial a}{\partial y}(x,t,y_{u_2})]z_1\\
\hspace{1.6cm}\displaystyle  - \frac{\partial^2 a}{\partial y^2}(x,t,y_{u_2})z_{u_2,v_1}z_{u_2,v_2}+ \frac{\partial^2 a}{\partial y^2}(x,t,y_{u_1})z_{u_1,v_1}z_{u_1,v_2}\ \  \mbox{in } Q,\vspace{2mm}\\  \partial_{\conormal_A}z + u_2z = (u_1 - u_2)z_1 - v_1z_{u_2,v_2} - v_2z_{u_2,v_1} + v_1z_{u_1,v_2} + v_2z_{u_1,v_1}\ \ \mbox{on }\Sigma, \\ z(0) = 0 \ \text{ in } \Omega,\end{array}\right.
\]
Let us estimate the right hand sides of the above equations. Using Assumption \ref{A2.2}, \eqref{E2.11} and \eqref{E2.13} we infer
\begin{align*}
\Big\Vert [\frac{\partial a}{\partial y}(x,t,y_{u_1}) &- \frac{\partial a}{\partial y}(x,t,y_{u_2})]z_1\Big\Vert _{L^\infty(Q)}\\
& \le C_{a,M_{\bar u}}C_{\bar u}M_{\bar u}\Vert u_2 - u_1\Vert _{L^p(\Sigma)}\Vert v_1\Vert _{L^p(\Sigma)}\Vert v_2\Vert _{L^p(\Sigma)}.
\end{align*}
The last two terms of the partial differential equation are rewritten as follows
\begin{align*}
\frac{\partial^2 a}{\partial y^2}(x,t,y_{u_2})z_{u_2,v_1}&z_{u_2,v_2} - \frac{\partial^2 a}{\partial y^2}(x,t,y_{u_1})z_{u_1,v_1}z_{u_1,v_2}\\
& = \Big[\frac{\partial^2 a}{\partial y^2}(x,t,y_{u_2}) - \frac{\partial^2 a}{\partial y^2}(x,t,y_{u_1})\Big]z_{u_2,v_1}z_{u_2,v_2}\\
&+\frac{\partial^2 a}{\partial y^2}(x,t,y_{u_1})[(z_{u_2,v_1} - z_{u_1,v_1})z_{u_2,v_2} + (z_{u_2,v_2} - z_{u_1,v_2})z_{u_1,v_1}].
\end{align*}
Using again Assumption \ref{A2.2}, \eqref{E2.11}, \eqref{E2.13}, and \eqref{E2.14} we get
\begin{align*}
&\Big\Vert \frac{\partial^2 a}{\partial y^2}(x,t,y_{u_2})z_{u_2,v_1}z_{u_2,v_2} - \frac{\partial^2 a}{\partial y^2}(x,t,y_{u_1})z_{u_1,v_1}z_{u_1,v_2}\Big\Vert _{L^\infty(Q)}\\
&\le K_{a,M_{\bar u}}C_{\bar u}\Vert u_2 - u_1\Vert _{L^p(\Sigma)}M_{\bar u}^2\Vert v_1\Vert _{L^p(\Sigma)}\Vert v_2\Vert _{L^p(\Sigma)}\\
& + C_{a,M_{\bar u}}2C_{\bar u}M_{\bar u}\Vert u_2 - u_1\Vert _{L^p(\Sigma)}\Vert v_1\Vert _{L^p(\Sigma)}\Vert v_2\Vert _{L^p(\Sigma)}.
\end{align*}
Finally we estimate the boundary terms
\begin{align}
&\Vert (u_1 - u_2)z_1 - v_1z_{u_2,v_2} - v_2z_{u_2,v_1} + v_1z_{u_1,v_2} + v_2z_{u_1,v_1}\Vert _{L^p(\Sigma)}\notag\\
& \le \Vert u_1 - u_2\Vert _{L^p(\Sigma)}\Vert z_1\Vert _{L^\infty(Q)} + \Vert z_{u_2,v_2} - z_{u_1,v_2}\Vert _{L^\infty(\Sigma)}\Vert v_1\Vert _{L^p(\Sigma)}\label{MME2.15}\\
& + \Vert z_{u_2,v_1} - z_{u_1,v_1}\Vert _{L^\infty(\Sigma)}\Vert v_2\Vert _{L^p(\Sigma)} \le [M_{\bar u} + 2C_{\bar u}]\Vert u_1 - u_2\Vert _{L^p(\Sigma)}\Vert v_1\Vert _{L^p(\Sigma)}\Vert v_2\Vert _{L^p(\Sigma)}.\notag
\end{align}
As a consequence of the established estimates we deduce the existence of $L_{\bar u}$ such that
\[
\Vert z\Vert _Y \le L_{\bar u}\Vert u_2 - u_1\Vert _{L^p(\Sigma)}\Vert v_1\Vert _{L^p(\Sigma)}\Vert v_2\Vert _{L^p(\Sigma)}\quad \forall v_1, v_2 \in L^p(\Sigma),
\]
which proves \eqref{E2.10}.
\end{proof}

\begin{remark}\label{R2.9}
Theorems \ref{T2.6} and \ref{T2.8} are valid if we replace the operators $A$ and $\partial_{\conormal_A}$ by the following ones:
\[
A^*\varphi = - \sum_{i,j=1}^\dimension \partial_{x_j} [ a_{ji}(x) \partial_{x_i} \varphi] \text{ and } \partial_{\conormal_{A^*}}\varphi = \sum_{i,j=1}^n a_{ij}(x) \partial_{x_i} \varphi(x) \conormal_j(x).
\]
Therefore, for every $\bar u\in\mA_0$ we obtain the existence of $\varepsilon_{\bar u}^*>0$ such that, for every $(f,h,z_0) \in L^r(0,T;L^s(\Omega)) \times L^{\hat r}(0,T;L^{\hat s}(\Gamma)) \times L^\infty(\Omega)$ and $u\in B_{\varepsilon_{\bar u}^*}(\bar u)$, the equation
\[
 \left\{\begin{array}{l}\displaystyle \frac{\partial z}{\partial t} + A^*z + \frac{\partial a}{\partial y}(x,t,y_u)z  = f \ \  \mbox{in } Q,\vspace{2mm}\\  \partial_{\conormal_{A^*}}z + uz = h \ \ \mbox{on }\Sigma, \ z(0) = z_0\ \ \mbox{in } \Omega\end{array}\right.
\]
has a unique solution $z \in Y$. Without loss of generality, in the sequel we will assume that $\varepsilon_{\bar u}\leq\varepsilon_{\bar u}^*$, so the above equation is uniquely solvable in $Y$ for all $u\in\mA$. Now, performing the change of variable $\varphi(t) = z(T-t)$ we deduce that the equation
\[
 \left\{\begin{array}{l}\displaystyle -\frac{\partial \varphi}{\partial t} + A^*\varphi + \frac{\partial a}{\partial y}(x,t,y_u)\varphi  = f \ \  \mbox{in } Q,\vspace{2mm}\\  \partial_{\conormal_{A^*}}\varphi + u\varphi = h \ \ \mbox{on }\Sigma, \ \varphi(T) = \varphi_T\ \ \mbox{in } \Omega\end{array}\right.
\]
has a unique solution $\varphi \in Y$ for every $(f,h,\varphi_T) \in L^r(0,T;L^s(\Omega)) \times L^{\hat r}(0,T;L^{\hat s}(\Gamma)) \times L^\infty(\Omega)$ and every $u\in\mA$.
\end{remark}

\section{Analysis of the control problem}
\label{S3}
\setcounter{equation}{0}

In this section, we formulate the assumptions on the cost functional $J$, study its differentiability, and prove the first- and second-order optimality conditions satisfied by a local minimizer.

\begin{assumption} \label{A3.1}
The mappings $L: Q \times \mathbb R \longrightarrow \mathbb R$ and $l: \Omega \times \mathbb R \longrightarrow \mathbb R$ are Carath\'eodory functions of class of $C^2$ with respect to the second variables.
Further the following properties hold for a.a. $x \in \Omega$, $x \in \Gamma$, and $t \in (0,T)$:
\begin{align*}
& \bullet L(\cdot,\cdot, 0) \in L^1(Q), l(\cdot,0) \in L^1(\Omega)\\
& \bullet \forall M >0, \  \exists \Psi_{L,M} \in L^r(0,T;L^s(\Omega)) \text{ and } \exists C_{l,M} \text{ such that }\\
& \hspace{1cm}\Big\vert \frac{\partial L}{\partial y}(x,t,y)\Big\vert  \leq \Psi_{L,M}(x,t) \text{ and } \Big\vert \frac{\partial^jl}{\partial y^j}(x,y)\Big\vert  \leq C_{l,M},\ j = 1, 2,\\
& \bullet \forall M >0,  \exists C_{L,M}\text{ such that } \Big\vert  \frac{\partial^2 L}{\partial y^2}(x,t,y)\Big\vert  \leq C_{L,M}, \\
& \bullet \forall M >0 \ \exists K_{L,M} \text{ such that }
\Big\vert \frac{\partial^2 L}{\partial y^2}(x,t,y_1) - \frac{\partial^2 L}{\partial y^2}(x,t,y_2) \Big\vert  \leq K_{L,M}\vert y_2 - y_1\vert ,\\
&\hspace{0.85cm} \text{ and } \exists K_{l,M} \ \text{ such that }\Big\vert \frac{\partial^2 l}{\partial y^2}(x,y_1) - \frac{\partial^2 l}{\partial y^2}(x,y_2) \Big\vert  \leq K_{l,M}\vert y_2 - y_1\vert ,
\end{align*}
for all $\vert y\vert , \vert y_1\vert , \vert y_2\vert  \leq M$.
All the above constants are independent of $x$ and $t$.
\end{assumption}
The following theorem states the differentiability properties of the minimizing functional.

\begin{theorem} \label{T3.2}
The functional $J: \mA \longrightarrow \mathbb R$ is of class $C^2$ and its derivatives are given by the expressions:
\begin{align}
& J'(u)v = \int_\Sigma (\tikhonov u - y_u \varphi_u)v  \dx\dt, \label{E3.1}\\
& J''(u)(v_1,v_2) = \int_Q \Big[\frac{\partial^2 L}{\partial y^2}(x,t,y_u) - \varphi_u  \frac{\partial^2 a}{\partial y^2}(x,t,y_u)\Big] z_{u,v_1} z_{u,v_2}\dx\dt \label{E3.2}\\
& \hspace{2cm}+ \int_\Omega\frac{\partial^2 l}{\partial y^2}(x,y_u(T))z_{u,v_1}(T) z_{u,v_2}(T)\dx\dt\notag\\
& \hspace{2cm}- \int_\Sigma \Big[ v_1z_{u,v_2} + v_2 z_{u,v_1} \Big] \varphi_u \dx\dt + \tikhonov \int_\Sigma v_1 v_2 \dx\dt,\notag
\end{align}
for all  $u \in \mA$ and all $v, v_1, v_2 \in L^p(\Sigma)$,
where $z_{u,v_i} = G'(u)v_i$, $i=1,2$, and $\varphi_u \in Y$ is the adjoint state, the unique solution of the equation
\begin{equation}
 \left\{\begin{array}{l} \displaystyle -\frac{\partial\varphi_u}{\partial t} + A^*\varphi_u + \frac{\partial a}{\partial y}(x,t,y_u)\varphi_u  = \frac{\partial L}{\partial y}(x,t,y_u) \ \  \mbox{in } Q,\vspace{2mm}\\\displaystyle  \partial_{\conormal_{A^*}} \varphi_u + u\varphi_u = 0\ \ \mbox{on }\Sigma,\ \varphi_u(T) = \frac{\partial l}{\partial y}(x,y_u(T))\ \ \mbox{in } \Omega. \end{array}\right.
\label{E3.3}
\end{equation}
\end{theorem}

\begin{proof}
It is a straightforward consequence of Theorem \ref{T2.8} and the chain rule that $J$ is of class $C^2$. Moreover, the derivation of the formulas \eqref{E3.1} and \eqref{E3.2} is standard. The existence and uniqueness of a solution $\varphi_u$ in $Y$ follows from Remark \ref{R2.9} and Assumption \ref{A3.1}.
\end{proof}

We define the mapping $\Psi:\mA \longrightarrow Y$ by $\Psi(u) = \varphi_u$ solution of \eqref{E3.3}. The next theorem establishes the differentiability of $\Psi$.

\begin{theorem}\label{T3.3}
The mapping $\Psi:\mA \longrightarrow Y$ defined by $\Psi(u) = \varphi_u$ is of class $C^1$ and for all $u \in \mA$ and $v \in L^p(\Sigma)$ $\eta_{u,v} = \Psi'(u)v$ is the unique solution in $Y$ of the equation
\begin{equation}
 \left\{\begin{array}{l} \displaystyle -\frac{\partial\eta}{\partial t} + A^*\eta + \frac{\partial a}{\partial y}(x,t,y_u)\eta  = \Big[\frac{\partial^2 L}{\partial y^2}(x,t,y_u) - \varphi_u  \frac{\partial^2 a}{\partial y^2}(x,t,y_u)\Big]z_{u,v} \,\mbox{in } Q,\vspace{2mm}\\\displaystyle  \partial_{\conormal_{A^*}} \eta + u\eta = -v\varphi_u\ \ \mbox{on }\Sigma,\ \eta(T) = \frac{\partial^2l}{\partial y^2}(x,y_u(T))z_{u,v}(T)\ \ \mbox{in } \Omega, \end{array}\right.
 \label{E3.4}
\end{equation}
where $z_{u,v} = G'(u)v$.
\end{theorem}

\begin{proof}
We apply the implicit function theorem similarly as we did in the proof of Theorem \ref{T2.8}. To this end we define
\[
\mY_{A^*} = \{\varphi \in Y : -\frac{\partial\varphi}{\partial t} + A^*\varphi \in L^r(0,T;L^s(\Omega))\text{ and } \partial_{\conormal_{A^*}}\varphi \in L^p(\Sigma)\},
\]
which is a Banach space when it is endowed with the graph norm. We also define the mapping
\begin{align*}
&\mG : \mY_{A^*} \times \mA \longrightarrow L^r(0,T;L^s(\Omega)) \times L^p(\Sigma) \times L^\infty(\Omega)\\
&\mG(\varphi,u) = \Big(-\frac{\partial\varphi}{\partial t} + A^*\varphi + \frac{\partial a}{\partial y}(x,t,y_u)\varphi - \frac{\partial L}{\partial y}(x,t,y_u),\\
&\hspace{2cm}\partial_{\conormal_{A^*}}\varphi + u\varphi,\varphi(T) - \frac{\partial l}{\partial y}(x,y_u(T))\Big).
\end{align*}
Due to the assumptions \ref{A2.2} and \ref{A3.1} and the differentiability properties of $G$ we get that $\mG$ is of class $C^1$ and
\begin{align*}
&\frac{\partial\mG}{\partial\varphi}(\varphi,u):\mY_{A^*} \longrightarrow L^r(0,T;L^s(\Omega)) \times L^p(\Sigma) \times L^\infty(\Omega),\\
&\frac{\partial\mG}{\partial\varphi}(\varphi,u)\eta = \Big(-\frac{\partial\eta}{\partial t} + A^*\eta + \frac{\partial a}{\partial y}(x,t,y_u)\eta,\partial_{\conormal_{A^*}}\eta + u\eta,\eta(T)\Big),
\end{align*}
is a linear and continuous mapping. Moreover, from Remark \ref{R2.9} we infer that the equation
\[
 \left\{\begin{array}{l} \displaystyle -\frac{\partial\eta}{\partial t} + A^*\eta + \frac{\partial a}{\partial y}(x,t,y_u)\eta  = f \ \ \mbox{in } Q,\vspace{2mm}\\\displaystyle  \partial_{\conormal_{A^*}} \eta + u\eta = h\ \ \mbox{on }\Sigma,\ \eta(T) = \varphi_T\ \ \mbox{in } \Omega, \end{array}\right.
\]
has a unique solution $\eta \in \mY_{A^*}$ for every $(f,h,\varphi_T) \in L^r(0,T;L^s(\Omega)) \times L^p(\Sigma) \times L^\infty(\Omega)$. Hence, applying the implicit function theorem we deduce that $\Psi$ is of class $C^1$ and the equation \eqref{E3.4} follows by differentiation of the identity $\mG(\Psi(u),u) = 0$.
\end{proof}

Combining \eqref{E3.2} with \eqref{E3.4} we deduce the following alternative formula for $J''(u)$.

\begin{corollary}\label{C3.4}
For every $v_1, v_2 \in L^p(\Sigma)$ and all $u \in \mA$, the following identities hold
\begin{align}
J''(u)(v_1,v_2)  &= \int_\Sigma\Big[ \tikhonov v_1 - (\varphi_u z_{u,v_1}  + y_u \eta_{u,v_1}) \Big] v_2 \dx\dt\label{E3.5}\\
&=  \int_\Sigma\Big[ \tikhonov v_2 - (\varphi_u z_{u,v_2}  + y_u \eta_{u,v_2}) \Big] v_1 \dx\dt.\notag
\end{align}
\end{corollary}

In the next theorem, we prove that the mappings $J''$ and $\Psi'$ are locally Lipschitz.

\begin{theorem}\label{T3.5}
For every $\bar u \in \mA$ there exist $r > 0$ and $K_{\bar u}$ such that
\begin{equation}
\Vert J''(u_2) - J''(u_1)\Vert  + \Vert \Psi'(u_2) - \Psi'(u_1)\Vert  \le K_{\bar u}\Vert u_2 - u_1\Vert _{L^p(\Sigma)}\quad \forall u_1, u_2 \in B_r(\bar u),
\label{E3.6}
\end{equation}
where $\Vert J''(u_2) - J''(u_1)\Vert $ and $\Vert \Psi'(u_2) - \Psi'(u_1)\Vert $ denote the norms in the space of continuous bilinear forms on $L^p(\Sigma)$ and in the space of linear and continuous mappings $\mathcal{L}(L^p(\Sigma),Y)$, respectively.
\end{theorem}

\begin{proof}
Since $\Psi:\mA \longrightarrow Y$ is of class $C^1$, it is enough to apply the generalized mean value theorem to deduce that for every $\bar u \in \mA$ there exists $r_0 > 0$ and $K_{\bar u,0}$ such that
\begin{align}
&\Vert \varphi_u\Vert _Y \le M_{\bar u,0} \text{ and } \Vert \eta_{u,v}\Vert _Y \le M_{\bar u,0}\Vert v\Vert _{L^p(\Sigma)}\quad \forall u \in B_{r_0}(\bar u)\text{ and } \forall v \in L^p(\Sigma),\label{E3.7}\\
&\Vert \varphi_{u_2} - \varphi_{u_1}\Vert _Y \le K_{\bar u,0}\Vert  u_2 - u_1\Vert _{L^p(\Sigma)} \quad \forall u_1, u_2 \in B_{r_0}(\bar u).
\label{E3.8}
\end{align}
Selecting $r \le r_0$ and smaller or equal to the radius introduced in \eqref{E2.10}, the Lipschitz property \eqref{E3.6} of $J''$ follows easily by using the Lipschitz properties in assumptions \ref{A2.2} and \ref{A3.1} along with \eqref{E2.11}--\eqref{E2.14} and \eqref{E3.7}--\eqref{E3.8}. To prove the Lipschitz property of $\Psi'$ we set $\eta_i = \Psi'(u_i)v$, $i = 1, 2$, for $v \in L^p(\Sigma)$ and $\eta = \eta_2 - \eta_1$. Then subtracting the equations satisfied by $\eta_2$ and $\eta_1$ we infer
\[
 \left\{\begin{array}{l} \displaystyle -\frac{\partial\eta}{\partial t} + A^*\eta + \frac{\partial a}{\partial y}(x,t,y_{u_2})\eta  = \Big[\frac{\partial a}{\partial y}(x,t,y_{u_1}) - \frac{\partial a}{\partial y}(x,t,y_{u_2})\Big]\eta_1\\
\hspace{3cm}\displaystyle +\Big[\frac{\partial^2 L}{\partial y^2}(x,t,y_{u_2}) - \varphi_{u_2}  \frac{\partial^2 a}{\partial y^2}(x,t,y_{u_2})\Big]z_{u_2,v}\\
\hspace{3cm}\displaystyle - \Big[\frac{\partial^2 L}{\partial y^2}(x,t,y_{u_1}) - \varphi_{u_1}  \frac{\partial^2 a}{\partial y^2}(x,t,y_{u_1})\Big]z_{u_1,v}\,\mbox{in } Q,\vspace{2mm}\\\displaystyle  \partial_{\conormal_{A^*}} \eta + u_2\eta = (u_1 - u_2)\eta_1 - v(\varphi_{u_2} - \varphi_{u_1})\ \ \mbox{on }\Sigma,\\
\displaystyle \eta(T) = \frac{\partial^2l}{\partial y^2}(x,y_{u_2}(T))z_{u_2,v}(T) - \frac{\partial^2l}{\partial y^2}(x,y_{u_1}(T))z_{u_1,v}(T)\ \ \mbox{in } \Omega. \end{array}\right.
\]
Using again Assumptions \ref{A2.2} and \ref{A3.1} along with \eqref{E2.11}--\eqref{E2.14} and \eqref{E3.7}--\eqref{E3.8}, arguing similarly as we did in the proof of Theorem \ref{T2.8}, it is easy to prove that the right hand sides of the above equation are bounded by $C\Vert u_2 -u_1\Vert _{L^p(\Sigma)}\Vert v\Vert _{L^p(\Sigma)}$, where $C$ is independent of $u_1, u_2$, and $v$. Hence, the above equation implies that $\Vert \eta\Vert _Y \le C'\Vert u_2 -u_1\Vert _{L^p(\Sigma)}\Vert v\Vert _{L^p(\Sigma)}$, which concludes the proof.
\end{proof}

The following theorem addresses the question the of existence of a solution for \Pb and the first-order optimality conditions satisfied by a local minimizer. In this paper, a local minimizer is intended in the $L^2(\Sigma)$ sense.

\begin{theorem}\label{T3.6}
Problem $\Pb$ has at least one solution. Moreover, if $\bar u \in \uad$ is a local minimizer of $\Pb$ then there exist $\bar y, \bar \varphi \in Y$ such that
\begin{align}
& \left\{\begin{array}{l} \displaystyle\frac{\partial \bar y}{\partial t} + A\bar y + a(x,t,\bar y) = 0\ \  \mbox{in } Q,\vspace{2mm}\\  \partial_{\conormal_A} \bar y + \bar u \bar y = g\ \ \mbox{on }\Sigma, \ \ y(0) = y_0\ \ \text{in } \Omega,\end{array}\right.
\label{E3.9} \\
& \left\{\begin{array}{l} \displaystyle -\frac{\partial \bar\varphi}{\partial t} + A^*\bar\varphi + \frac{\partial a}{\partial y}(x,t,\bar y)\bar\varphi = \frac{\partial L}{\partial y}(x,t,\bar y) \ \  \mbox{in } Q,\vspace{2mm}\\\displaystyle  \partial_{\conormal_{A^*}} \bar\varphi + \bar u\bar\varphi = 0\ \ \mbox{on }\Sigma,\ \ \bar\varphi(T) =  \frac{\partial l}{\partial y}(x,\bar y(T))\ \ \text{in } \Omega,\end{array}\right. \label{E3.10} \\
& \tcr{\int_\Sigma(\tikhonov \bar u - \bar y\bar\varphi)(u-\bar u)\dx\geq 0 \ \forall u\in\uad. } \label{E3.11}
\end{align}
\end{theorem}
\tcr{Let us remind that condition \eqref{E3.11} can be equivalently written as
\begin{equation}\label{E3.11b} 
\bar u(x,t) = \proj_{[\umin, \umax]} \Big(\frac{1}{\tikhonov} \bar y(x,t) \bar \varphi(x,t)\Big)\qquad \text{a.e.~on }\Sigma.
\end{equation}
}

\begin{proof}
The existence of optimal solutions follows applying the direct method of the calculus of variations. This is carried out easily by recalling that for every sequence $\{u_k\}_{k=1}^\infty\subset \uad$ such that $u_k \stackrel{*}{\rightharpoonup} \bar u$ in $L^\infty(\Sigma)$, the associated states $\{y_k\}_{k = 1}^\infty$ converge to $\bar y = y_{\bar u}$ weakly in $W(0,T)$ and strongly in $L^\infty(Q)$; see Theorem \ref{T2.6}.

The optimality system \eqref{E3.9}--\eqref{E3.11} is proved as usual utilizing the convexity of $\uad$ and the expression for $J'(\bar u)$ given by \eqref{E3.1}.
\end{proof}

Now, we carry out the second-order analysis of \Pb. To this end, in the rest of the paper, we will denote by $\bar u$ an element of $\uad$ with associated state $\bar y$ and adjoint state $\bar\varphi$ satisfying the first-order optimality conditions \eqref{E3.9}--\eqref{E3.11}.

\begin{remark}\label{R3.9}Let us observe that the mappings $J'(\bar u)$ and $J''(\bar u)$ can be extended to continuous linear and bilinear forms, respectively, on $L^2(\Sigma)$ still given by the formulas \eqref{E3.1} and \eqref{E3.2}. Indeed, this is obvious for the extension of $J'(\bar u)$ because $\bar y, \bar\varphi \in L^\infty(Q)$. For the extension of $J''(\bar u)$ it is enough to observe that for every $v \in L^2(\Sigma)$ the equation \eqref{E2.8} with $u$ replaced by $\bar u$ has a unique solution $z_{\bar u,v} \in W(0,T)$ and hence $z_{\bar u,v}\in L^2(Q)$, $z_{\bar u,v}(T)\in L^2(\Omega)$, and  ${z_{\bar u,v\vert }}_{\Sigma}\in L^{2 + \frac{2}{d}}(\Sigma)$. This follows from Theorem \ref{AT1}, the inclusion $W(0,T)\subset C([0,T];L^2(\Omega))$, and Theorem \ref{AT2}.
\end{remark}

Associated with $\bar u$ we define the cone of critical directions
\[
C_{\bar u} = \{v \in L^{2}(\Sigma) : v(x,t) = 0 \text{ if } \tikhonov \bar u(x,t)  -  \bar y(x,t) \bar\varphi(x,t) \ne 0 \text{ and \eqref{E3.12} holds}\},
\]
\begin{equation}
v(x,t)  \left\{ \begin{array}{ll} \geq 0 & \text{ if } \bar u(x,t) = \umin, \\ \leq 0 & \text{ if } \bar u(x,t) = \umax. \end{array} \right.
\label{E3.12}
\end{equation}
Obviously, the above equalities and inequalities must be understood almost everywhere on $\Sigma$. To carry out the second-order analysis we need to introduce an additional assumption.

\begin{assumption}\label{A3.10}
In the sequel we assume that at least one of the following two properties holds:
\begin{align*}
& \text{I - } \frac{\partial^2l}{\partial y^2}(x,y) \ge 0\quad \forall y \in \mathbb{R},\\
& \text{II - } \Gamma \text{ is a $C^{1,1}$ manifold and } a_{ij} \in C^{0,1}(\bar\Omega)\quad \forall 1 \le i, j \le \dimension.
\end{align*}
\end{assumption}

The following technical lemma is used in the next two theorems.

\begin{lemma}\label{L3.11}
Let $\{u_k\}_{k = 1}^\infty \subset \uad$ and $\{v_k\}_{k = 1}^\infty \subset L^2(\Sigma)$ satisfy that $u_k \rightharpoonup u$ and $v_k \rightharpoonup v$ in $L^2(\Sigma)$. Then, the following properties hold:
\begin{enumerate}
\item $z_k = G'(u_k)v_k \rightharpoonup z = G'(u)v$ in $W(0,T)$.
\item $z_k \to z$ in $L^2(Q)$ and $z_{k\mid_\Sigma} \to z_{\mid_\Sigma}$ in $L^2(\Sigma)$.
\item if Assumption \ref{A3.10}-{II} holds, then $z_k(T) \to z(T)$ in $L^2(\Omega)$.
\item $J''(u)v^2 \le \liminf_{k \to \infty}J''(u_k)v_k^2$.
\end{enumerate}
\end{lemma}

\begin{proof}
First we observe that the boundedness of $\uad$ in $L^\infty(\Sigma)$ implies that $u_k \rightharpoonup u$ in $L^p(\Sigma)$ and, consequently $y_{u_k} \to y_u$ in $L^\infty(Q)$ and $y_{u_k} \rightharpoonup y_u$ in $W(0,T)$; see Theorem \ref{T2.6}. Then, from Theorem \ref{AT1} we infer the boundedness of $\{z_k\}_{k = 1}^\infty$ in $W(0,T)$. Using these facts it is immediate to pass to the limit in the state equation satisfied by $z_k$ and deduce that $z_k \rightharpoonup z$ in $W(0,T)$. Thus {\it 1} holds. From the compactness of the embedding $W(0,T) \subset L^2(Q)$ we deduce the first convergence of {\it 2}. The second convergence of {\it 2} follows from Theorem \ref{AT2}.

Now, we prove {\it 3}. For every $k \ge 1$ and for $\sigma \in (\frac{1}{2},1)$ we define the linear mapping
\begin{align*}
&f_k:L^2(0,T;H^\sigma(\Omega)) \longrightarrow \mathbb{R},\\
&\langle f_k,w\rangle = \int_\Sigma(u_kz_k + v_ky_{u_k})w\dx\dt.
\end{align*}
Using that the trace of $H^\sigma(\Omega)$ is contained in $L^2(\Gamma)$ for every $\sigma > \frac{1}{2}$, from the estimates proved for $y_{u_k}$ as well as the boundedness of $\{u_k\}_{k = 1}^\infty$ in $L^\infty(\Sigma)$ and $\{z_k\}_{k = 1}^\infty$ and $\{v_k\}_{k = 1}^\infty$ in $L^2(\Sigma)$, we deduce the existence of a constant $C$ such that
\begin{align*}
\vert \langle f_k,w\rangle\vert  &\le \Big(\Vert u_k\Vert _{L^\infty(\Sigma)}\Vert z_k\Vert _{L^2(\Sigma)} + \Vert y_{u_k}\Vert _{L^\infty(\Sigma)}\Vert v_k\Vert _{L^2(\Sigma)}\Big)\Vert w\Vert _{L^2(\Sigma)}\\
&\le C \Vert w\Vert _{L^2(0,T;H^\sigma(\Gamma))}.
\end{align*}
Hence, due to the regularity assumed in \ref{A3.10}-II, from the maximal parabolic regularity we infer that $\{z_k\}_{k = 1}^\infty$ is a bounded sequence in  $H^1(0,T;H^\sigma(\Omega)^*) \cap L^2(0,T;H^{2-\sigma}(\Omega))$. Then, we apply \cite[Theorem 3]{Amann2001} with $\frac{1}{2} < s < 1 - \frac{\sigma}{2}$ and $\theta = \frac{\sigma}{2}$ and obtain
\begin{align*}
H^1(0,T;H^\sigma(\Omega)^*) \cap L^2(0,T;H^{2-\sigma}(\Omega)) &\subset C^{s - \frac{1}{2}}([0,T];(H^\sigma(\Omega)^*,H^{2 - \sigma}(\Omega))_{\frac{\sigma}{2},2})\\
&= C([0,T];L^2(\Omega)),
\end{align*}
with compact embedding. This implies that $z_k \to z$ in $C([0,T];L^2(\Omega))$ and, hence, $z_k(T) \to z(T)$ in $L^2(\Omega)$.

Finally, we prove {\it 4}. From \eqref{E3.2} we get that
\begin{align*}
J''(u_k)v_k^2 &= \int_Q \Big[\frac{\partial^2 L}{\partial y^2}(x,t,y_{u_k}) - \varphi_{u_k}  \frac{\partial^2 a}{\partial y^2}(x,t,y_{u_k})\Big] z_k^2\dx\dt - \int_\Sigma 2v_kz_k\varphi_{u_k} \dx\dt\\
&+ \int_\Omega\frac{\partial^2 l}{\partial y^2}(x,y_{u_k}(T))z_k^2(T) \dx\dt  + \tikhonov \int_\Sigma v_k^2 \dx\dt.
\end{align*}
From Theorem \ref{T2.6} we know that $y_{u_k}  \to y_u$ in $L^\infty(\Sigma)$ and $y_{u_k}(T) \to y_u(T)$ in $L^\infty(\Omega)$. Using this in \eqref{E3.3} along with the assumption \ref{A3.1} we deduce that $\varphi_{u_k} \to  \varphi_u$ in $W(0,T) \cap L^\infty(Q)$. Hence, using the first convergence of {\it 2}, we have
\begin{align*}
\lim_{k \to \infty}&\int_Q \Big[\frac{\partial^2 L}{\partial y^2}(x,t,y_{u_k}) - \varphi_{u_k}  \frac{\partial^2 a}{\partial y^2}(x,t,y_{u_k})\Big] z_k^2\dx\dt\\
&= \int_Q \Big[\frac{\partial^2 L}{\partial y^2}(x,t,y_u) - \varphi_u \frac{\partial^2 a}{\partial y^2}(x,t,y_u)\Big] z^2\dx\dt.
\end{align*}
From the second convergence of {\it 2} and the inequality $\Vert y_{\mid_\Gamma}\Vert_{L^\infty(\Gamma)} \le \Vert y\Vert_{L^\infty(\Omega)}$ for all $y \in H^1(\Omega) \cap L^\infty(\Omega)$, we infer that $z_k\varphi_{u_k} \to z\varphi_u$ in $L^2(\Sigma)$. Therefore, we get the convergence of the second integral of $J''(u_k)v_k^2$:
\[
\lim_{k \to \infty}\int_\Sigma 2v_kz_k\varphi_{u_k} \dx\dt = \int_\Sigma 2vz\varphi_u \dx\dt.
\]
For the third integral we proceed as follows
\begin{align*}
\int_\Omega\frac{\partial^2 l}{\partial y^2}(x,y_{u_k}(T))z_k^2(T) \dx\dt  &= \int_\Omega\Big[\frac{\partial^2 l}{\partial y^2}(x,y_{u_k}(T)) - \frac{\partial^2 l}{\partial y^2}(x,y_u(T))\Big]z_k^2(T) \dx\dt \\
&+ \int_\Omega\frac{\partial^2 l}{\partial y^2}(x,y(T))z_k^2(T) \dx\dt.
\end{align*}
From Assumption \ref{A3.1} and the boundedness of $\{z_k(T)\}_{k = 1}^\infty$ in $L^2(\Omega)$ we deduce
\begin{align*}
&\int_\Omega\Big\vert \frac{\partial^2 l}{\partial y^2}(x,y_{u_k}(T)) - \frac{\partial^2 l}{\partial y^2}(x,y_u(T))\Big\vert  z_k^2(T) \dx\dt\\
& \le K_{l,M}\Vert y_{u_k} - y_u\Vert _{L^\infty(Q)}\Vert z_k(T)\Vert _{L^2(\Omega)} \stackrel{k  \to \infty}{\longrightarrow} 0.
\end{align*}
If Assumption \ref{A3.10}-I holds, then we have

\[
\int_\Omega\frac{\partial^2 l}{\partial y^2}(x,y(T))z^2(T) \dx\dt \le \liminf_{k \to \infty}\int_\Omega\frac{\partial^2 l}{\partial y^2}(x,y(T))z_k^2(T) \dx\dt.
\]
In the case that Assumption \ref{A3.10}-II holds, then from {\it 3} we have an equality in the above expression even replacing $\liminf$ by $\lim$. Finally, since $$\Vert v\Vert ^2_{L^2(\Sigma)} \le \displaystyle\liminf_{k \to \infty}\Vert v_k\Vert ^2_{L^2(\Sigma)},$$ we conclude that {\it 4} holds.
\end{proof}

Using Lemma \ref{L3.11} the proof of the following theorem is standard; see, e.g.
\cite[Theorem 2.3]{Casas-Troltzsch2012}.

\begin{theorem}\label{T3.11}
If $\bar u$ is a local minimizer of $\Pb$, then $J''(\bar u)v^2 \geq 0 \ \forall v \in C_{\bar u}$ holds. Conversely, if $\bar u \in \uad$ satisfies the first-order optimality conditions \eqref{E3.9}--\eqref{E3.11} and $J''(\bar u)v^2 > 0$ $\forall v \in C_{\bar u} \setminus\{0\}$, then there exist $\varepsilon >0$ and $\delta >0$ such that

\begin{equation}
J(\bar u) + \frac{\delta}{2} \Vert u - \bar u\Vert ^2_{L^2(\Sigma)} \leq J(u)\ \forall u \in \uad \text{ with } \Vert u - \bar u\Vert _{L^2(\Sigma)} \leq \varepsilon.
\label{E3.13}
\end{equation}
\end{theorem}

\begin{definition} \label{D3.12}
Let us denote
\[
\Sigma_{\bar u} = \{(x,t) \in \Sigma : \bar u(x,t) \in \{ \umin, \umax \} \text{ and } \tikhonov \bar u(x,t) - \bar y(x,t)\bar\varphi (x,t) = 0 \}.
\]
We say that the strict complementarity condition is satisfied at $\bar u$ if $ \vert \Sigma_{\bar u}\vert  = 0$, where $\vert  \cdot \vert $ stands for the $\dimension$ dimensional Lebesgue measure on $\Sigma$.
\end{definition}

For every $\tau \geq 0$, we define the subspace of $L^2(\Sigma)$
\[
 E^\tau_{\bar u} = \{v \in L^2(\Sigma) : v(x,t) = 0 \text{ if } \vert \tikhonov\bar u(x,t)  -  \bar y (x,t) \bar\varphi (x,t) \vert  > \tau\}.
\]
\begin{theorem}\label{T3.13}
Assume that the strict complementarity condition is satisfied at $\bar u$. Then, the following properties hold:
\begin{enumerate}
\item $E^0_{\bar u} = C_{\bar u}$,
\item If $\bar u$ satisfies the second-order optimality condition $J''(\bar u)v^2 >0$ $\forall v \in C_{\bar u} \setminus\{0\}$, then $\exists \tau >0$ and $\nu >0$ such that
\end{enumerate}
\begin{equation}
J''(\bar u) v^2 \geq \nu \Vert v\Vert ^2_{L^2(\Sigma)} \quad \forall v \in E^\tau_{\bar u}.
\label{E3.14}
\end{equation}
\end{theorem}

\begin{proof}
The identity $E^0_{\bar u} = C_{\bar u}$ is a straightforward consequence of \eqref{E3.11} and the strict complementarity condition satisfied by $\bar u$. Let prove the second property.  We will proceed by contradiction: suppose \eqref{E3.14} is false. Then, there exists a sequence $\{v_k\}_{k=1}^\infty\subset L^2(\Sigma)$ such that $v_k\in E_{\bar u}^{1/k}$ and $J''(\bar u)v_k^2<\frac{1}{k}\Vert v_k\Vert ^2_{L^2(\Sigma)}$. Of course, we can assume that $\Vert v_k\Vert _{L^2(\Omega)}=1$, otherwise it is enough to divide $v_k$ by its $L^2(\Sigma)$-norm. Thus we have
\begin{equation}
v_k\in E_{\bar u}^{1/k},\ \Vert v_k\Vert _{L^2(\Sigma)}=1, \text{ and } J''(\bar u)v_k^2<\frac{1}{k}.
\label{E3.15}
\end{equation}
Then, for a subsequence, denoted in the same way, there exists $v\in L^2(\Sigma)$ such that $v_k \rightharpoonup v$ weakly in $L^2(\Sigma)$. We observe that $v \in C_{\bar u}$. Indeed, for every $\varepsilon > 0$ we set
\[
\Sigma^\varepsilon = \{(x,t) \in\Sigma : v(x,t) \neq 0 \ \text{ and }\ \vert \tikhonov\bar u(x,t)  -  \bar y (x,t) \bar\varphi (x,t)\vert  >\varepsilon\}.
\]
Due to $v_k$ vanishes in $\Sigma^\varepsilon$ for every $k > \frac{1}{\varepsilon}$, its weak limit $v$ vanishes as well in $\Sigma^\varepsilon$. As $\varepsilon > 0$ is arbitrary, we conclude that $v \in E^0_{\bar u} = C_{\bar u}$. Now, using Lemma \ref{L3.11}-{\it 4} and \eqref{E3.15} we get $0 \le J''(\bar u)v^2\leq \liminf_{k\to\infty} J''(\bar u)v_k^2 \leq \limsup_{k\to\infty} J''(\bar u)v_k^2 = 0$. From the assumed second-order optimality condition, we deduce that $v=0$ and hence $z_{\bar u,v} = 0$. Now, using again the convergences proved in Lemma \ref{L3.11} we infer
\begin{align*}
0 = \lim_{k\to\infty}J''(\bar u)v_k^2 & = \lim_{k\to\infty}\int_Q\left( \frac{\partial^2 L}{\partial y^2}(x,t,\bar y)-\bar\varphi \frac{\partial^2a}{\partial y^2}(x,t,\bar y)\right)z_{\bar u,v_k}^2\\
& - \lim_{k\to\infty}\int_\Sigma 2v_kz_{\bar u,v_k}\bar\varphi \dx\dt\dx\dt\\
&+ \liminf_{k\to\infty}\Big(\int_\Omega\frac{\partial^2 l}{\partial y^2}(x,\bar y(T))z_{\bar u,v_k}(T)^2 \dx + \tikhonov \int_\Sigma v_k^2 \dx\dt\Big)\\
&\ge  \liminf_{k\to\infty}\tikhonov \int_\Sigma v_k^2 \dx\dt = \tikhonov,
\end{align*}
which contradicts the fact that $\tikhonov>0$.
\end{proof}

\section{Lagrangian formulation of the optimality conditions}\label{S4}
\setcounter{equation}{0}
In order to study the convergence of the SQP method analyzed in Section \ref{S5}, it is convenient to introduce the Lagrangian function associated with the control problem \Pb.    In the rest of the paper we increase the value of $p$ by assuming 
\begin{equation}\label{MME5.1}p =  2(\dimension+1).\end{equation} 
This assumption is necessary to prove \tcr{Theorems \ref{AT3} and \ref{AT5}}, and the auxiliary estimate \eqref{MME4.10} in Lemma \ref{L4.3}. \tcr{We will also assume, without loss of generality, that the $\varepsilon_{\bar u}$ introduced in Theorem \ref{T2.8} is taken small enough so that the conditions in Remark \ref{AR4} hold.}   In the sequel, we will adopt the following notation:
\begin{align*}
&W = \mY_{A}\times \mY_{A^*}\times L^p(\Sigma),\ \tcr{\mW = \mY_{A}\times \mY_{A^*} \times \mA},\ 
\wad = \mY_A \times \mY_{A^*} \times \uad,
\\&
W_q = L^\infty(Q) \times L^\infty(Q) \times L^q(\Sigma) \ \text{ for } 1 \le q \le \infty.
\end{align*}
For $w = (y,\varphi,u)\in W_q$ and $\rho > 0$, we denote $\Vert w\Vert_q = \Vert y \Vert_{L^\infty(Q)} +  \Vert \varphi \Vert_{L^\infty(Q)} +
 \Vert u \Vert_{L^q(\Sigma)}$ and  
\[
B^q_\rho(w) = \{w'\in W_q:\ \Vert w'- w\Vert _{W_q}\le\rho\}.
\]
We consider the functional $\mJ:\tcr{\mY_{A}} \times L^p(\Sigma)\longrightarrow \mathbb{R}$ and the Lagrangian function $\mL:W\longrightarrow\mathbb{R}$ given by
\begin{align*}
\mJ(y,u) = & \int_Q L(x,t,y(x,t)) \dx\dt + \int_\Omega l(x,y(x,T))\dx + \frac{\tikhonov}{2} \int_\Sigma u^2(x,t) \dx\dt,\\
\mL(y,\varphi,u) = &  \mJ(y,u) - \int_Q(\partial_t y + A y + a(x,t,y))\varphi\dx\dt \\ & -\int_\Sigma(\partial_{n_A} y + uy -g)\varphi\dx\dt -
\int_\Omega (y(0)-y_0)\varphi(0)\dx.
\end{align*}
For the record, we compute the derivatives of the Lagrangian with respect to $y$ and $u$. For every $w=(y,\varphi,u)\in W$, $z\in\mY_{A}$ and $v\in L^p(\Sigma)$, we have that
\begin{align*}
&\partial_y \mathcal{L}(w) z = \int_Q \partial_y L(x,t)z\dx\dt + \int_\Omega \partial_y l(x,y(T)) z(T)\dx \\
& \ \ -\int_Q\big(\partial_t z + A z + \partial_y a(x,t,y)z\big)\varphi\dx\dt - \int_\Sigma \big(\partial_{n_{A}}z+uz\big)\varphi\dx\dt - \int_\Omega z(0)\varphi(0)\dx\\
&\partial^2_{yy}\mathcal{L}(w) z^2 = \int_Q\hspace{-5pt} \left( \partial^2_{yy}L(x,t,y) -\varphi\partial^2_{yy}a(x,t,y)\right) z^2\dx\dt +\hspace{-5pt} \int_\Omega \partial^2_{yy}l(x,y(T))z(T)^2\dx\\
&\partial^2_{yu}\mathcal{L}(w) (z,v) = -\int_\Sigma v z \varphi\dx\dt,\\
&\partial_u \mathcal{L}(w) v  = \tikhonov\int_\Sigma (u-y\varphi)v\dx\dt,\
 \partial^2_{uu}\mathcal{L}(w) v^2 = \tikhonov \int_\Sigma v^2\dx\dt.
\end{align*}
We have
\[D^2_{(y,u)}\mL(w)(z,v)^2 =  \partial^2_{yy}\mathcal{L}(y,\varphi,u) z^2 + 2\partial^2_{yu}\mathcal{L}(y,\varphi,u)(z,v) + \partial^2_{uu}\mathcal{L}(y,\varphi,u) v^2.
\]
Note that for every $w\in W$, $D^2_{(y,u)}\mL(w)$ can be extended to a continuous bilinear form in $W(0,T)\times L^2(\Sigma)$. Using \tcr{Theorem \ref{AT3}, we} deduce that for every $w\in\tcr{\mW}$ and any $v\in L^2(\Sigma)$, there exist unique functions $\zw_{w,v}\in W(0,T)$ and $\etaw_{w,v}\in W(0,T)$ solutions, respectively, to the equations
\begin{equation}\label{E4.4}
\left\{\begin{array}{l} \displaystyle\frac{\partial\zw}{\partial t} + A\zw +  \frac{\partial a}{\partial y}(x,t,y)\zw = 0\ \  \mbox{in } Q,\vspace{2mm}\\  \partial_{\conormal_A}\zw + u\zw =- y v \ \ \mbox{on }\Sigma,\ \zw(0) = 0 \ \text{ in } \Omega, \end{array}\right.
\end{equation}
\begin{equation}
 \left\{\begin{array}{l} \displaystyle -\frac{\partial\etaw}{\partial t} + A^*\etaw + \frac{\partial a}{\partial y}(x,t,y)\etaw  = \Big[\frac{\partial^2 L}{\partial y^2}(x,t,y) - \varphi  \frac{\partial^2 a}{\partial y^2}(x,t,y)\Big]\zw_{w,v}  \,\mbox{in } Q,\vspace{2mm}\\\displaystyle  \partial_{\conormal_{A^*}} \etaw + u\etaw = -v\varphi\ \ \mbox{on }\Sigma,\ \etaw(T) = \frac{\partial^2l}{\partial y^2}(x,y(T))\zw_{w,v}(T)\ \ \mbox{in } \Omega. \end{array}\right.
 \label{E4.1}
\end{equation}
Moreover, with the help of assumptions \ref{A2.2} and \ref{A3.1}, with infer \tcr{from \eqref{AE6} the} existence of a constant $c_{w}>0$ that depends on $\Vert w\Vert _{W_p}$  such that
\begin{equation}\label{E4.5}
\Vert \zw_{w,v}\Vert _{W(0,T)} + \Vert \etaw_{w,v}\Vert _{W(0,T)} \leq c_{w} \Vert v\Vert _{L^2(\Sigma)}\quad \forall v \in L^2(\Sigma).
\end{equation}
Using these equations, a standard computation leads to the expression
\begin{align}
D^2_{(y,u)}\mL(w)(\zw_{w,v},v)^2 =  \int_\Sigma (\tikhonov v-y\etaw_{w,v}-\zw_{w,v}\varphi)v\dx\dt\quad \forall v\in L^2(\Sigma).  \label{E4.3}
 \end{align}
For $\bar w=(\bar y,\bar\varphi,\bar u)\in\wad$ satisfying the first-order optimality system \eqref{E3.9}--\eqref{E3.11} and $v\in L^2(\Sigma)$ we have that $\zw_{\bar w,v} = G'(\bar u)v = z_{\bar u,v}$ and $\etaw_{\bar w,v} = \Psi'(\bar u) v = \eta_{\bar u,v}$. Then, from \eqref{E3.5} we deduce the following identity:
\begin{equation}
D^2_{(y,u)}\mL (\bar w)(\zw_{\bar w,v},v)^2 = J''(\bar u) v^2.
\label{E4.5A}
\end{equation}
Hence, the condition \eqref{E3.14} can be written as
\begin{equation}\label{eq::sscLa}
D^2_{(y,u)}\mL(\bar w)(\zw_{\bar w,v} ,v)^2 \geq \nu\Vert v\Vert ^2_{L^2(\Sigma)}\ \forall v\in E^{\tau}_{\bar u}.
\end{equation}
Our main result in this section is the following theorem.
\begin{theorem}\label{T4.1} Let $\bar w=(\bar y,\bar\varphi,\bar u)\in\wad$ satisfy the first-order optimality conditions \eqref{E3.9}--\eqref{E3.11}. Suppose that the strict complementarity condition given in Definition \ref{D3.12} is satisfied at $\bar u$ and that $J''(\bar u)v^2 > 0$ for all $v\in C_{\bar u}\setminus\{0\}$. Then, there exist $\rho_{\textsc{ssc}}>0$ such that for every $w=(y,\varphi,u)\in B^p_{\rho_{\textsc{ssc}}}(\bar w) \tcr{\cap \mW}$,
\[
D^2_{(y,u)}\mL(w)(\zw_{w,v},v)^2 \geq \frac{\nu}{2}\Vert v\Vert ^2_{L^2(\Sigma)}\ \forall v\in E^\tau_{\bar u},
\]
where $\nu>0$ and $\tau>0$ are the numbers introduced in Theorem \ref{T3.13}.
\end{theorem}
For the proof of this theorem we use two technical lemmas that are stated and proved below. We introduce some notation. Set $\rho_0=\min\{1,\varepsilon_{\bar u}\}$, 
where $\varepsilon_{\bar u}$ 
the one introduced in the proof of Theorem \ref{T2.8}. 
 From \tcr{Theorem \ref{AT3} and} assumptions \ref{A2.2} and \ref{A2.3} we infer the existence of a constant $c_{\bar w}$ such that \eqref{E4.5} holds with $c_w$ replaced by $c_{\bar w}$ for all $w\in B^p_{\rho_0}(\bar w) \tcr{\cap \mW}$. Finally, we denote $M=\Vert \bar w\Vert_{W_p}+\rho_0$.

\begin{lemma}\label{L4.2}There exists $K_{\bar w}^1>0$ such that for every $w\in B^p_{\rho_0}(\bar w) \tcr{\cap \mW}$ the following estimate holds:
\begin{align}
\big\vert[D^2_{(y,u)}\mL(\bar w) - D^2_{(y,u)}\mL(w)](\zw_{\bar w,v},v)^2\big\vert \leq K_{\bar w}^1\Vert \bar w-w\Vert _{W_p}\Vert v\Vert _{L^2(\Sigma)}^2\ \forall v\in L^2(\Sigma).
\label{E4.7}
\end{align}
\end{lemma}
\begin{proof}
Using assumptions \ref{A2.2} and \ref{A3.1}, together with \eqref{AE5}, \eqref{AE6a} and \eqref{E4.5} we obtain
\begin{align*}
 \big\vert D^2 &_{(y,u)}\mL(\bar w)(\zw_{\bar w,v},v)^2 - D^2_{(y,u)}\mL(w)(\zw_{\bar w,v},v)^2\big\vert \\
 = & \Big\vert
 \int_Q \left[\frac{\partial^2 L}{\partial y^2}(x,t,\bar y)-\bar \varphi \frac{\partial^2 a}{\partial y^2}(x,t,\bar y)
 -
 \left(\frac{\partial^2 L}{\partial y^2}(x,t,y)-\varphi \frac{\partial^2 a}{\partial y^2}(x,t,y)\right)\right]
 \zw_{\bar w,v}^2\dx\dt\\
  & + \int_\Omega \left[\frac{\partial^2 l}{\partial y^2}(x,\bar y) -
   \frac{\partial^2 l}{\partial y^2}(x,y)
   \right] \zw_{\bar w,v}(T)^2\dx - 2 \int_\Sigma v \zw_{\bar w,v}(\bar\varphi-\varphi)\dx\dt\Big\vert \\
  \leq &
  \int_Q\left\vert \frac{\partial^2 L}{\partial y^2}(x,t,\bar y)-
  \frac{\partial^2 L}{\partial y^2}(x,t, y)\right\vert  \zw_{\bar w,v}^2\dx\dt  + 
   \int_Q \left\vert\bar \varphi \frac{\partial^2 a}{\partial y^2}(x,t,\bar y) - \varphi \frac{\partial^2 a}{\partial y^2}(x,t,\bar y) \right\vert  \zw_{\bar w,v}^2\dx\dt \\
   & + 
   \int_Q \left\vert \varphi \frac{\partial^2 a}{\partial y^2}(x,t,\bar y) - \varphi \frac{\partial^2 a}{\partial y^2}(x,t, y) \right\vert  \zw_{\bar w,v}^2\dx\dt \\
    & + \int_\Omega \left\vert\frac{\partial^2 l}{\partial y^2}(x,\bar y) -
   \frac{\partial^2 l}{\partial y^2}(x,y)
   \right\vert \zw_{\bar w,v}(T)^2\dx
   + 2 \int_\Sigma \vert v \zw_{\bar w,v}\vert \vert\bar\varphi-\varphi\vert\dx\dt \\
    \leq &\Big( K_{L,M} \Vert \bar y-y\Vert_{L^\infty(Q)}  + C_{a,M}\Vert\bar\varphi-\varphi\Vert_{L^\infty(Q)} +
    M K_{a,M}  \Vert \bar y-y\Vert_{L^\infty(Q)} 
    \Big)
    \Vert \zw_{\bar w,v} \Vert_{L^2(Q)}^2 \\
    & + K_{l,M} \Vert \bar y-y\Vert_{L^\infty(Q)} \Vert \zw_{\bar w,v}(T) \Vert_{L^2(\Omega)}^2    
    + 2 \Vert\bar\varphi-\varphi\Vert_{L^\infty(Q)} \Vert v \Vert_{L^2(\Sigma)} \Vert \zw_{\bar w,v} \Vert_{L^2(\Sigma)} \\
    \leq &
    \Big( (K_{L,M} + C_{a,M} + M K_{a,M} + K_{l,M}C_\Omega) c_{\bar  w}^2 +2 c_{\bar w} C_\Sigma\Big) \Vert \bar w-w\Vert_{W_p} \Vert v\Vert ^2_{L^2(\Sigma)}
\end{align*}
and the proof is complete.
\end{proof}

\begin{lemma}\label{L4.3} There exists $K_{\bar w}^2>0$ such that for all  $w\in  B^p_{\rho_0}(\bar w) \tcr{\cap \mW}$ the following estimate holds:
\begin{equation}
\vert D^2_{(y,u)}\mL(w)(\zw_{w,v},v)^2-D^2_{(y,u)}\mL(w)(\zw_{\bar w,v},v)^2\vert
\leq K_{\bar w}^2\Vert w-\bar w\Vert _{W_p} \Vert v\Vert _{L^2(\Sigma)}^2\ \forall v\in L^2(\Sigma).
\label{E4.10}
\end{equation}
\end{lemma}
\begin{proof}
First, let us show the existence of $K_{\bar w}''>0$ such that
\begin{equation}\label{E4.8}
\Vert \zw_{\bar w,v}-\zw_{w,v}\Vert _{W(0,T)} \leq K_{\bar w}''\Vert w-\bar w\Vert _{W_p}\Vert v\Vert _{L^2(\Sigma)}.
\end{equation}
The difference $\zw=\zw_{\bar w,v}-\zw_{w,v}$ satisfies
\[
\left\{\begin{array}{l} \displaystyle\frac{\partial \zw}{\partial t} + A\zw +  \frac{\partial a}{\partial y}(x,t,\bar y)\zw = \big[\frac{\partial a}{\partial y}(x,t,y)-\frac{\partial a}{\partial y}(x,t,\bar y)\big] \zw_{w,v} \ \  \mbox{in } Q,\vspace{2mm}\\  \partial_{\conormal_A} \zw + \bar u\zw  =(u-\bar u) \zw_{w,v} - (\bar y-y) v \ \ \mbox{on }\Sigma,\ \zw(0) = 0 \ \text{ in } \Omega. \end{array}\right.
\]
From \tcr{Theorem \ref{AT3},} Assumption \ref{A2.2}, and \eqref{E4.5}, we obtain the existence of a constant $c>0$ such that
\begin{align*}
 \Vert\zeta\Vert_{W(0,T)} \leq &
c\Big( \Vert y-\bar y\Vert_{L^\infty(Q)}
\Vert \zw_{w,v} \Vert_{L^2(Q)}
+\Vert (u-\bar u) \zw_{w,v}\Vert_{L^2(\Sigma)}
+\Vert (\bar y-y)v \Vert_{L^2(\Sigma)}\Big)\\
&\le c\Big(\Vert y-\bar y\Vert_{L^\infty(Q)}[c_{\bar w}+ 1]\Vert v\Vert_{L^2(\Sigma)}] +\Vert (u-\bar u) \zw_{w,v}\Vert_{L^2(\Sigma)}\Big).
\end{align*}
To estimate the term $\Vert (u-\bar u) \zw_{w,v} \Vert _{L^2(\Sigma)}$ we proceed as follows:
take $q$ such that $1/q+1/p = 1/2$. Recalling that $p = 2(\dimension + 1)$, we have that $q = 2+2/d$. Then, from \eqref{AE5} and \eqref{E4.5} we infer that
$\Vert \zw_{w,v} \Vert _{L^q(\Sigma)}\leq c_{\Sigma}c_{\bar w} \Vert v\Vert _{L^2(\Sigma)}$. Using H\"older inequality with $\frac{p}{2}$ and $\frac{q}{2}$ we get that
\begin{align}\label{MME4.10}
\Vert (u-\bar u) \zw_{w,v}\Vert _{L^2(\Sigma)} &\le \Vert u-\bar u \Vert _{L^p(\Sigma)}\Vert \zw_{w,v} \Vert _{L^q(\Sigma)}
\le c_{\Sigma}c_{\bar w}\Vert u-\bar u \Vert _{L^p(\Sigma)} \Vert v\Vert _{L^2(\Sigma)}
\end{align}
and \eqref{E4.8} follows.
Next we use the definition of $D^2_{(y,u)}\mL(w)$, assumptions \ref{A2.2} and \ref{A2.3}, \eqref{E4.8}, \eqref{AE5}, and \eqref{E4.5} to obtain
\begin{align*}
  \vert & D^2_{(y,u)}\mL(w)(\zw_{w,v},v)^2-D^2_{(y,u)}\mL(w)(\zw_{\bar w,v},v)^2\vert   \\
  = & \Big\vert
  \int_Q\left[\frac{\partial^2 L}{\partial y^2}(x,t,y)-\varphi \frac{\partial^2 a}{\partial y^2}(x,t,y)\right]
  ( \zw_{\bar w,v}^2-\zw_{w,v}^2 ) \dx\dt - 2 \int_\Sigma v \varphi ( \zw_{\bar w,v}-\zw_{w,v} ) \dx\dt\Big\vert\\
  \leq & (C_{L,M}+M C_{a,M})\Vert  \zw_{\bar w,v}-\zw_{w,v} \Vert_{L^2(Q)} \Vert  \zw_{\bar w,v}+\zw_{w,v} \Vert_{L^2(Q)} + 2 M \Vert v\Vert_{L^2(\Sigma)} \Vert  \zw_{\bar w,v}-\zw_{w,v} \Vert_{L^2(\Sigma)}\\
  \leq & \Big((C_{L,M}+M C_{a,M}) K_{\bar w}'' 2 c_{\bar w}+
  2M C_\Sigma K''_{\bar w}\Big) \Vert w-\bar w\Vert _{W_p} \Vert v\Vert _{L^2(\Sigma)}^2
\end{align*}
and the proof is complete.
\end{proof}

\begin{proof}{\em of Theorem \ref{T4.1}}.
Let $\nu>0$ and $\tau>0$ be the ones given in Theorem \ref{T3.13} and set 
\[
\rho_{\textsc{ssc}} = \min\big\{\rho_0, \frac{\nu}{2 K_{\bar w}}\big\},
\]
where $\rho_0$ is defined after the statement of Theorem \ref{T4.1} and $K_{\bar w} = K_{\bar w}^1+K_{\bar w}^2$ with $K_{\bar w}^1$ and $K_{\bar w}^2$ introduced in Lemmas \ref{L4.2} and \ref{L4.3}, respectively. Take $w\in B^{\tcr{p}}_{\rho_{\textsc{ssc}}}(\bar w) \tcr{\cap \mW}$ and $v\in E^\tau_{\bar u}$.
Using the Lipsichtz property of $D^2_{(y,u)}\mL$ obtained in \eqref{E4.7}, we have that
\begin{align*}
  \vert D^2_{(y,u)}\mL(\bar w)(\zw_{\bar w,v},v)^2 -  D^2_{(y,u)}\mL(w)(\zw_{\bar w,v},v)^2\vert
 \leq K^1_{\bar w}\rho_{\textsc{ssc}} \Vert v\Vert _{L^2(\Sigma)}^2.
\end{align*}
From equation \eqref{E4.10}, we have that
\begin{align*}
  \vert D^2_{(y,u)}\mL( w)(\zw_{w,v},v)^2 - & D^2_{(y,u)}\mL(w)(\zw_{\bar w,v},v)^2\vert
  \leq K^2_{\bar w} \rho_{\textsc{ssc}} \Vert v\Vert _{L^2(\Sigma)}^2.
  \end{align*}
Using the triangle inequality and the two previous estimates, we obtain
\[ \vert D^2_{(y,u)}\mL(\bar w)(\zw_{\bar w,v},v)^2 -  D^2_{(y,u)}\mL(w)(\zw_{w,v},v)^2\vert \leq
 K_{\bar w}\rho_{\textsc{ssc}} \Vert v\Vert _{L^2(\Sigma)}^2\leq \frac{\nu}{2}\Vert v\Vert _{L^2(\Sigma)}^2.\]
Then, \eqref{E4.7} is a straightforward consequence of the above inequality and \eqref{eq::sscLa}.
\end{proof}

\section{Analysis of the SQP method to solve (P)}\label{S5}
\setcounter{equation}{0}
In this section, we will prove quadratic convergence of the SQP method. We detail it in Algorithm \ref{Alg0}. For the convenience of the reader, we notice that the radius $r^\star$ is introduced in Corollary \ref{R5.12} and functions $\zw_{w_n,v}$ and $\etaw_{w_n,v}$ are respectively solutions of equations \eqref{E4.4} and \eqref{E4.1}.
\LinesNumbered
\begin{algorithm2e}
\caption{SQP Method to solve \Pb.}\label{Alg0}
\DontPrintSemicolon
Initialize. Choose $(y_0,\varphi_0,u_0)\in \wad\cap B^2_{r^\star}(\bar w)$. Set $n=0$.\;
\For{$n\geq 0$}{
Compute $\zw_0^n $ solving the linear equation \eqref{E5.10}
\;
Compute $\etaw^n_0$ solving the linear equation \eqref{E5.11}\;
Find a local solution of the constrained quadratic program
\begin{align*}
\mathrm{(}Q_n\mathrm{)}\quad & \min_{v\in\uad-\{u_n\}} \frac{1}{2}\int_\Sigma (kv-y_n\etaw_{w_n,v}-\zw_{w_n,v}\varphi_n)v\dx\dt \\
& \qquad \qquad + \int_\Sigma (\tikhonov u_n - y_n \etaw_0^n-\zw_0^n\varphi_n -y_n\varphi_n )v\dx\dt
\end{align*}\;
Name $v_n$ the solution, $\zw_n = \zw_{w_n,v_n}+\zw_0^n$,  and $\etaw_{n} = \etaw_{w_n,v_n}+\etaw_0^n$.
\;
Set $y_{n+1}=y_n+\zw_{n}$, $\varphi_{n+1} = \varphi_n+\etaw_{n}$, $u_{n+1}=u_n+v_{n}$ and $n=n+1$.\;
}
\end{algorithm2e}

In the rest of the results of the paper, we will make the following assumption.
\begin{assumption}\label{A5.1}
The triplet $\bar w = (\bar y,\bar\varphi,\bar u)\in\wad$ satisfies the necessary first-order optimality conditions \eqref{E3.9}--\eqref{E3.11}, the strict complementarity condition given in Definition \ref{D3.12} and the second-order optimality condition $J''(\bar u)v^2 >0$ $\forall v \in C_{\bar u} \setminus\{0\}$.
\end{assumption}

We are going to write the first-order optimality system as a generalized equation.
Let us consider the space
\begin{align*}
E =& L^r(0,T;L^s(\Omega))\times (L^p(\Sigma) + L^{\hat r}(0,T;L^{\hat s}(\Gamma))) \times L^\infty(\Omega)\\
 & \times L^r(0,T;L^s(\Omega))\times L^p(\Sigma) \times L^\infty(\Omega) \times L^p(\Sigma)
 \end{align*}
and the mapping $F:W\longrightarrow E$ defined by
\begin{align*}
  F_1(w) =  & \partial_t y + A y + a(\cdot,\cdot,y)
  \\
  F_2(w) = & \partial_{n_A} y + uy -g 
  \\
  F_3(w) = & y(0)-y_0
  \\
  F_4(w) = & -\partial_t\varphi + A^*\varphi + \partial_y a(\cdot,\cdot,y)\varphi -\partial_y L(\cdot,\cdot,y) 
  \\
  F_5(w) = &  \partial_{n_A^*} \varphi + u\varphi 
  \\
  F_6(w) = & \varphi(T)-\partial_y l(\cdot,y(T)) 
  \\
  F_7(w) = & \tikhonov u -y\varphi
\end{align*}
For any control $u\in L^p(\Sigma)$, we introduce the normal cone of $\uad$ at $u$ as
\[N(u) = \left\{\begin{array}{cl}
\{v\in L^p(\Sigma):\ \int_\Sigma v(\vt-u)\dx\dt\leq 0\ \forall \vt\in\uad\} & \text{ if }u\in\uad \\ \\
\emptyset& \text{ if }u\not\in\uad
\end{array}\right.
\]
and for any $w\in W$, we define
\[\normalConeW(w) = \big( \{0\},\{0\},\{0\},\{0\},\{0\},\{0\},N(u)\big)\in 2^E.\]
With this notation $\bar w = (\bar y,\bar\varphi,\bar u)\in \wad$ is a solution of the first-order optimality system \eqref{E3.9}--\eqref{E3.11} if and only if
\begin{equation}\label{E5.1}
  0\in F(\bar w) + \normalConeW(\bar w).
\end{equation}
Given $w_n\in \tcr{\mW}$, the generalized Newton method defines $w_{n+1}\in\wad$ as a solution of
\begin{equation}\label{E5.2}
0\in F(w_n) + F'(w_n)(w_{n+1}-w_n) + \normalConeW(w_{n+1}).
\end{equation}
For any point $w=(y,\varphi,u)\in W$ and direction $\direction=(\zw,\etaw,v)\in W$, we have
\begin{align*}
  F_1'(w)\direction = & \partial_t\zw + A\zw +\partial_y a(\cdot,\cdot,y)\zw \\
  F_2'(w)\direction = & \partial_{n_A}\zw + u\zw + vy \\
  F_3'(w)\direction = & \zw(0) \\
  F_4'(w)\direction = & -\partial_t\etaw + A^*\etaw + \partial_y a(\cdot,\cdot,y)\etaw - \big(\partial^2_{yy}L(\cdot,\cdot,y) -\varphi \partial^2_{yy} a(\cdot,\cdot,y) \big)\zw \\
  F_5'(w)\direction = & \partial_{n_{A^*}}\etaw + u\etaw + v\varphi \\
  F_6'(w)\direction = & \etaw(T) - \partial^2_{yy}l(\cdot,y(T))\zw(T) \\
  F_7'(w)\direction = & \tikhonov v -y\etaw-\zw\varphi.
\end{align*}
From now on, we will shorten $a(\cdot,\cdot,y)=a(y)$, and so on. Setting $(\zw,\etaw,v) = w_{n+1} - w_n = (y_{n+1}-y_n,\varphi_{n+1}-\varphi_n,u_{n+1} - u_n)$, the generalized Newton step \eqref{E5.2} can be written in detail as follows:
\begin{equation}\label{E5.3}
\left\{\begin{array}{rcll}
  \partial_t\zw + A\zw + \partial_y a(y_n)\zw& = &
  -\big(\partial_t y_n + A y_n + a(y_n)\big)&\text{ in }Q \\
  \partial_{n_A}\zw + u_n\zw & =  &- v y_n -\big(\partial_{n_A} y_n + u_n y_n\big)+g&\text{ on }\Sigma \\
  \zw(0)& =  & -\big( y_n(0)-y_0 \big)&\text{ on }\Omega
\end{array}\right.
\end{equation}
\begin{equation}\label{E5.4}
\left\{\!\!\begin{array}{rcll}
-\partial_t\etaw + A^*\etaw+\partial_y a(y_n)\etaw & = &
-\big( -\partial_t\varphi_n + A^*\varphi_n +\partial_y a(y_n)\varphi_n\big) \\ &&+ \partial_y L(y_n)
 + \big(\partial^2_{yy} L(y_n) - \varphi_n\partial^2_{yy} a(y_n)\big)\zw & \text{ in } Q\\ \partial_{n_{A^*}}\etaw + u_n\etaw & =  &- v \varphi_n -\big(\partial_{n_{A^*}} \varphi_n + u_n \varphi_n\big)&\text{ on }\Sigma \\
\etaw(T) & = & -\varphi_n(T) + \partial_y l(y_n(T)) + \partial^2_{yy} l (y_n(T)) \zw(T)&\text{ on }\Omega
\end{array}\right.
\end{equation}

\[
\int_\Sigma \big(
\tikhonov v - y_n\etaw - \zw\varphi_n + \tikhonov u_n - y_n\varphi_n \big)
(u - u_{n+1})\dx\dt\geq 0\ \forall u \in\uad.
\]
Writing in the previous expression $u - u_{n+1} =
(u - u_n) - (u_{n+1} - u_n) =(u - u_{n}) - v$ and noting that $u$ is a dumb variable, we will write the variational inequality as
\begin{equation}\label{E5.5}
 \int_\Sigma \big(\tikhonov v - y_n\etaw - \zw\varphi_n + \tikhonov u_n -y_n\varphi_n \big)
(u - v)\dx\dt\geq 0\ \forall u \in\uad-\{u_n\},
\end{equation}
where
\[\uad-\{u_n\} = \{u \in L^p(\Sigma):\ \alpha-u_n(x,t)\leq u(x,t)\leq \beta-u_n(x,t)\text{ for a.a. }(x,t)\in\Sigma\}.
\]
Relations \eqref{E5.3}, \eqref{E5.4} and \eqref{E5.5} constitute the optimality system of the control-constrained linear-quadratic optimal control problem $(Q_n)$ formulated in the following lemma.
\begin{lemma}\label{L5.6} Given $w_n = (y_n,\varphi_n,u_n)\in \tcr{\mW}$, let $\zw_0^n\in \mY_A$ and $\etaw_0^n\in\mY_{A^*}$ be the solutions of
\begin{equation}\label{E5.10}
\left\{\begin{array}{rcll}
  \partial_t\zw + A\zw + \partial_y a(y_n)\zw& = &
  -\big(\partial_t y_n + A y_n + a(y_n)\big)&\text{ in }Q, \\
  \partial_{n_A}\zw + u_n\zw & =  & -\big(\partial_{n_A} y_n + u_n y_n\big)+g&\text{ on }\Sigma, \\
  \zw(0)& =  & -\big( y_n(0)-y_0 \big)&\text{ on }\Omega,
\end{array}\right.
\end{equation}
\begin{equation}\label{E5.11}
\!\!\!\left\{\!\!\begin{array}{rcll}
-\partial_t\etaw + A^*\etaw+\partial_y a(y_n)\etaw & = &
-\big( -\partial_t\varphi_n + A^*\varphi_n +\partial_y a(y_n)\varphi_n\big) \\ &&+ \partial_y L(y_n)
 + \big(\partial^2_{yy} L(y_n) - \varphi_n\partial^2_{yy} a(y_n)\big)\zw_0^n & \text{ in } Q,\\
\partial_{n_{A^*}}\etaw + u_n\etaw & =  & -\big(\partial_{n_{A^*}} \varphi_n + u_n \varphi_n\big)&\text{ on }\Sigma, \\
\etaw(T) & = & -\varphi_n(T) + \partial_y l(y_n(T)) + \partial^2_{yy} l (y_n(T)) \zw_0^n(T)&\text{ on }\Omega,
\end{array}\right.
\end{equation}
respectively, and consider the quadratic functional
\begin{align*}
\mathcal{Q}_n(v) = &\frac{1}{2}D^2_{(y,u)}\mL(w_n)(\zw_{w_n,v},v)^2 + \int_\Sigma (\tikhonov u_n - y_n\varphi_n - y_n \etaw_0^n-\zw_0^n\varphi_n  )v\dx\dt,
\end{align*}
and the optimization problem
\[
(Q_n)\quad \left\{\begin{array}{l}\min \mathcal{Q}_n(v) \\ v\in \uad-\{u_n\}.\end{array}\right.
\]
This problem has at least one solution. Let $v\in \uad-\{u_n\}$ be a local solution
and define $\zw=\zw_{w_n,v}+\zw_0^n$ and $\etaw = \etaw_{w_n,v}+\etaw_0^n$, where $\zw_{w_n,v}$ and $\etaw_{w_n,v}$ are defined according to the equations \eqref{E4.4} and \eqref{E4.1}, respectively. Then, $(\zw,\etaw,v)$ satisfies \eqref{E5.3}--\eqref{E5.5}.
\end{lemma}
\begin{proof}
First, we prove that the optimization problem $(Q_n)$ has at least one solution. Since $\uad-\{u_n\}$ is a bounded, closed, and convex subset of $L^\infty(\Sigma)$, it is enough to prove that for any sequence $\{v_k\}_{k = 1}^\infty \subset \uad-\{u_n\}$ converging weakly$^*$ to $v$ in $L^\infty(\Sigma)$ we have the weak$^*$ lower semicontinuity of $\mathcal{Q}_n$: $\mathcal{Q}_n(v) \le \liminf_{k \to \infty}\mathcal{Q}_n(v_k)$. This is a straightforward consequence of the expression \eqref{E4.3} and Theorem \ref{AT2}.

Let $v\in\uad-\{u_n\}$ be a local solution of $(Q_n)$. It is clear that $\zw=\zw_{w_n,v}+\zeta_0^n$ and $\etaw = \etaw_{w_n,v}+\etaw_0^n$ satisfy \eqref{E5.3} and \eqref{E5.4}, respectively. Using again \eqref{E4.3} and noting that $\zw_{w_n,v}$ and $\etaw_{w_n,v}$ are linear functions of $v$, see equations \eqref{E4.4} and \eqref{E4.1}, we infer with Lemma \ref{AL4} proved below that
\begin{align*}
\mathcal{Q}_n'(v)u & = \int_\Sigma \tikhonov vu\dx\dt\\
& - \frac{1}{2}\Big(\int_\Sigma(y_n\etaw_{w_n,v} + \zw_{w_n,v}\varphi_n\big)u\dx\dt + \int_\Sigma(y_n\etaw_{w_n,u} + \zw_{w_n,u}\varphi_n\big)v\dx\dt\Big)\\
&+ \int_\Sigma\big( \tikhonov u_n - y_n\etaw_0^n - \zw^n_0\varphi_n  -y_n\varphi_n \big)
u\dx\dt\\
&= \int_\Sigma \big(\tikhonov v - y_n\etaw - \zw\varphi_n + \tikhonov u_n -y_n\varphi_n \big)u \dx\dt,
\end{align*}
and \eqref{E5.5} is the first-order necessary optimality condition of $(Q_n)$.
\end{proof}

\begin{lemma}\label{AL4} For all $w=(y,\varphi,u)\in\tcr{\mW}$, and all $v_1,v_2\in L^2(\Sigma)$ the following equality holds:
  \[\int_\Sigma (y\etaw_{w,v_1}+\varphi\zw_{w,v_1})v_2\dx\dt = 
  \int_\Sigma (y\etaw_{w,v_2}+\varphi\zw_{w,v_2})v_1\dx\dt. \]
\end{lemma}
\begin{proof}
Using \eqref{E4.4}, doing integration by parts once in time and twice in space, and applying \eqref{E4.1} and the boundary conditions in \eqref{E4.1} and \eqref{E4.4}, we get
\begin{align*}
  0 =  & \int_Q\left(
  \frac{\partial \zw_{w,v_2}}{\partial t} + A \zw_{w,v_2} + \frac{\partial a}{\partial y}(x,t,y) \zw_{w,v_2}\right) \etaw_{w,v_1} \dx\dt\\
  =  &  \int_Q\left(
  -\frac{\partial \etaw_{w,v_1}}{\partial t} + A \etaw_{w,v_1} + \frac{\partial a}{\partial y}(x,t,y) \etaw_{w,v_1}\right) \zw_{w,v_2} \dx\dt \\
  & + \int_\Omega \etaw_{w,v_1}(T) \zw_{w,v_2}(T) \dx 
    - \int_\Sigma \etaw_{w,v_1}\partial_{\conormal_A} \zw_{w,v_2}\dx\dt 
    + \int_\Sigma \zw_{w,v_2}\partial_{\conormal_{A^*}} \etaw_{w,v_1}\dx\dt\\
  = & \int_Q \Big[\frac{\partial^2 L}{\partial y^2}(x,t,y) - \varphi  \frac{\partial^2 a}{\partial y^2}(x,t,y)\Big]\zw_{w,v_1}\zw_{w,v_2}\dx\dt + \int_\Omega \frac{\partial^2 l}{\partial y^2}(x, y(T))\zw_{w,v_1}(T)\zw_{w,v_2}(T)\dx\\
  & - \int_\Sigma \etaw_{w,v_1}\partial_{\conormal_A} \zw_{w,v_2}\dx\dt 
  - \int_\Sigma u\etaw_{w,v_1}\zw_{w,v_2}\dx\dt 
  - \int_\Sigma v_1\varphi\zw_{w,v_2}\dx\dt \\
= & \int_Q \Big[\frac{\partial^2 L}{\partial y^2}(x,t,y) - \varphi  \frac{\partial^2 a}{\partial y^2}(x,t,y)\Big]\zw_{w,v_1}\zw_{w,v_2}\dx\dt + \int_\Omega \frac{\partial^2 l}{\partial y^2}(x, y(T))\zw_{w,v_1}(T)\zw_{w,v_2}(T)\dx\\
  & + \int_\Sigma v_2 y \etaw_{w,v_1}
  - \int_\Sigma v_1\varphi\zw_{w,v_2}\dx\dt  \Rightarrow \int_\Sigma v_2 y \etaw_{w,v_1} = \int_\Sigma v_1\varphi\zw_{w,v_2}\dx\dt.
\end{align*}
Swapping the roles of $v_1$ and $v_2$ we obtain
\[
\int_\Sigma v_1 y \etaw_{w,v_2} = \int_\Sigma v_2\varphi\zw_{w,v_1}\dx\dt.
\]
The last two identities prove the the statement of the Lemma.
\end{proof}

Following \cite{Troltz1999}, we define now
\begin{align*}
  \Sigma_{\bar u}^{\tau+} =  & \{(x,t)\in\Sigma:\ \tikhonov\bar u(x,t) -\bar y(x,t)\bar\varphi(x,t) > \tau\}, \\
  \Sigma_{\bar u}^{\tau-} =  & \{(x,t)\in\Sigma:\ \tikhonov\bar u(x,t) -\bar y(x,t)\bar\varphi(x,t) < -\tau\}, \\
  \Sigma_{\bar u}^{\tau} =  & \Sigma_{\bar u}^{\tau+} \cup \Sigma_{\bar u}^{\tau-},\ \widehat{\uad} = \{u\in\uad:\  u(x,t) = \bar u(x,t)\text{ for a.e. }(x,t)\in \Sigma^\tau_{\bar u}\},
\end{align*}
where $\tau>0$ was introduced in Theorem \ref{T3.13}. In the same spirit, we define the set $\widehat{\wad}=\mY_A\times \mY_{A^*} \times \widehat{\uad}$, and for $w =(y,\varphi,u)\in\widehat{\wad}$, we denote by $\widehat{N}(u)$ and $\widehat{\normalConeW}(w)$ the normal cones of $\widehat{\uad}$ at $u$ and of $\widehat{W_{\mathrm{ad}}}$ at $w$. It is clear that $\bar w$ satisfies
\[
0\in F(\bar w)+\widehat{\normalConeW}(\bar w),
\]
and the generalized Newton step to solve this problem is
\begin{equation}\label{E5.12}
0\in F(w_n)+F'(w_n)(w_{n+1}-w_n)+\widehat{\normalConeW}( w_{n+1}).
\end{equation}

\begin{lemma}\label{L5.8}
For every $w_n\in B^p_{\rho_{\textsc{ssc}}}(\bar w) \tcr{\cap \mW}$ the generalized equation
  \eqref{E5.12}
  has a unique solution  $ w_{n+1}\in \widehat{\wad}$.
\end{lemma}
\begin{proof}
    As in Lemma \ref{L5.6}, we can prove that the generalized equation \eqref{E5.12} is the first-order optimality system of the following constrained quadratic problem:
\[
(\widehat{Q_n})\quad \left\{\begin{array}{l}\min \mathcal{Q}_n(v) \\ v\in \widehat{\uad}-\{u_n\}
\end{array}\right.
\]
The proof of the existence of a solution of $(\widehat{Q_n})$ is the same as the one given in the proof of Lemma \ref{L5.6} for $(Q_n)$. Let us prove that this problem is strictly convex and consequently it has a unique local minimizer that is the global minimizer, characterized by the first-order optimality conditions \eqref{E5.12}. Consider $v_a, v_b \in \widehat{\uad}-\{u_n\}$ with $v_a\neq v_b$. Since $\mathcal{Q}''(v)u^2=D^2_{(y,u)}\mL(w_n)(\zw_{w_n,u},u)^2$ for all $u\in L^2(\Sigma)$ and $v_a-v_b\in E^\tau_{\bar u}$, we deduce from Theorem \ref{T4.1} that $\mathcal{Q}_n''(v)(v_a-v_b)^2 >0$. Then, we conclude by a standard argument that $\mathcal{Q}_n$ is strictly convex in $\widehat{\uad}-\{u_n\}$; see \cite[Theorem 7.4-3]{Ciarlet1982}. Hence, the result follows.
\end{proof}
For any $w_n\in \tcr{\mW}$ and $w\in W$ we consider $\delta_{n}(w)\in E$ defined as
\begin{equation}\label{MME5.15}
  \delta_{n}(w) = F(\bar w)-F(w_n) + F'(\bar w)(w-\bar w) - F'(w_n)(w-w_n).
\end{equation}
It is straightforward to check that the unique solution $w_{n+1}\in\widehat{W}_{ad}$ of $(\widehat{Q_n})$ satisfies
\begin{equation}\label{MME5.14}
  \delta_{n}(w_{n+1})\in F(\bar w) + F'(\bar w)(w_{n+1}-\bar w)+\widehat{\normalConeW}(w_{n+1}).
\end{equation}
We now prove several results about the perturbations $\delta_n(w)$ and the perturbed problem
\begin{equation}\label{eq::pertuebedproblem}
  \delta\in F(\bar w) + F'(\bar w)(w-\bar w)+\widehat{\normalConeW}(w).
\end{equation}
We define the Banach space
\[
E_\infty = [L^r(0,T;L^s(\Omega))\times L^p(\Sigma)  \times L^\infty(\Omega)]^2 \times L^\infty(\Sigma).
\]
Let us recall that $\rho_0=\min\{1,\varepsilon_{\bar u}\}$ with $\varepsilon_{\bar u}$ introduced in the proof of Theorem \ref{T2.8}.
\begin{lemma}\label{L5.4}The following properties hold:
\begin{itemize}
\item[\textup{(i)}] For all $(w_n,w) \in \tcr{\mW} \times W$, we have that $\delta_n(w)\in E_\infty$.
 \item[\textup{(ii)}] There exists a constant $L_{\bar w}>0$ such that
\[
\Vert [F'(\tilde w)-F'(w_n)](w_a-w_b)\Vert _{E_{\infty}}\leq L_{\bar w}
     \Vert \tilde w-w_n\Vert _{W_p} \Vert w_a-w_b\Vert _{W_p}
\]
is satisfied for all $\tilde w, w_n\in B^p_{\rho_0}(\bar w) \tcr{\cap \mW}$ and all $w_a,w_b\in W$.
 \item[\textup{(iii)}]For all $w_n\in  B^p_{\rho_0}(\bar w)\tcr{\cap \mW} $ and all $w_a,w_b\in W$, the following estimate holds
     \[\Vert \delta_n(w_a)-\delta_n(w_b)\Vert _{E_{\infty}}\leq L_{\bar w}
     \Vert \bar w-w_n\Vert _{W_p} \Vert w_a-w_b\Vert _{W_p}.\]
 \item[\textup{(iv)}] Suppose $w_n\in  B^p_{\rho_0}(\bar w) \tcr{\cap \mW}$ and $w\in W$, then we have
\[
\Vert \delta_n(w)\Vert _{E_\infty}\leq \frac{L_{\bar w}}{2}\Vert \bar w-w_n\Vert _{W_p}^2 + L_{\bar w}  \Vert \bar w-w_n\Vert _{W_p}\Vert \bar w-w\Vert _{W_p}.
\]
\end{itemize}
\end{lemma}
\begin{proof}
(i) It is clear that $\delta_n(w)\in E$. Since the definitions of $E$ and of $E_\infty$ only differ in the second and seventh spaces used in the cartesian product, we only need to check these spaces. After performing the necessary computations, it becomes clear that this property holds.

\noindent(ii)  Denote $\delta= [F'(\tilde w)-F'(w_n)](w_a-w_b)$ and $w_a-w_b=(\zw,\etaw,v)$. After some straightforward computations, we have
\begin{align*}
\delta_1 &= (\partial_y a(\tilde y)-\partial_y a(y_n))\zw,\  \delta_2 =  (\tilde u-u_n)\zw + (\tilde y-y_n)v,\ \delta_3 =  0, \\
\delta_4 &= (\partial_y a(\tilde y)-\partial_y a(y_n))\etaw -
  (\partial^2_{yy} L(\tilde y)-\partial^2_{yy} L(y_n))\zw \\
& -[ (\tilde\varphi - \varphi_n)\partial^2_{yy} a(\tilde y)
  + \varphi_n(\partial^2_{yy}a(\tilde y) - \partial^2_{yy}a(y_n))]\zw,\\
\delta_5 &= (\tilde u - u_n)\etaw + v(\tilde\varphi-\varphi_n),\ \delta_6 = -(\partial^2_{yy} l(\tilde y(T)) - \partial^2_{yy} l(y_n(T)))\zw(T),\\
\delta_7 &= -[(\tilde y-y_n)\etaw +(\tilde\varphi-\varphi_n)\zw].
\end{align*}
Now we use the Lipschitz properties of $a,L,l$ and their derivatives stated in assumptions \ref{A2.2} and \ref{A3.1}. We will denote, $k_{\bar w}$ a generic constant that will depend on the data of the problem and $\Vert \bar y\Vert _{L^\infty(Q)}$ and $\Vert \bar\varphi\Vert _{L^\infty(Q)}$.
Finally, we notice that $\zw,\etaw\in W(0,T)\cap L^\infty(Q)$ implies that both traces on $\Omega$ and $\Sigma$ are essentially bounded.
\begin{align*}
  \Vert \delta_1\Vert _{L^\infty(Q)} \leq  &
  k_{\bar w} \Vert \tilde y-y_n\Vert _{L^\infty(Q)} \Vert \zw\Vert _{L^\infty(Q)}\\
  \Vert \delta_2\Vert _{L^p(\Sigma)} \leq  &
  k_{\bar w} (\Vert \tilde u -u_n\Vert _{L^p(\Sigma)} \Vert \zw\Vert _{L^\infty(Q)} + \Vert \tilde y - y_n\Vert _{L^\infty(Q)} \Vert v\Vert _{L^p(\Sigma)}) \\
  \Vert \delta_3\Vert _{L^\infty(\Omega)} = & 0 \\
  \Vert \delta_4 \Vert _{L^\infty(Q)}  \leq &
  k_{\bar w} \big[\Vert \tilde y - y_n\Vert _{L^\infty(Q)} \Vert \etaw\Vert _{L^\infty(Q)} \\
  &\quad +  (\Vert \tilde y - y_n\Vert _{L^\infty(Q)} +\Vert \tilde\varphi - \varphi_n\Vert _{L^\infty(Q)}) \Vert \zw\Vert _{L^\infty(Q)}\big] \\
  \Vert \delta_5\Vert _{L^p(\Sigma)} \leq &
  k_{\bar w}
  (\Vert \tilde u -u_n\Vert _{L^p(\Sigma)} \Vert \etaw\Vert _{L^\infty(\Sigma)} + \Vert \tilde\varphi - \varphi_n\Vert _{L^\infty(\Sigma)} \Vert v\Vert _{L^p(\Sigma)})
   \\
  \Vert \delta_6\Vert _{L^\infty(\Omega)} \leq &
  k_{\bar w}
  (\Vert \tilde y - y_n\Vert _{L^\infty(Q)}\Vert \zw\Vert _{L^\infty(Q)}
  \\
  \Vert \delta_7\Vert _{L^\infty(\Sigma)} \leq &  k_{\bar w}
  (\Vert \tilde y-y_n\Vert _{L^\infty(Q)} \Vert \etaw\Vert _{L^\infty(Q)} +
  \Vert \tilde\varphi-\varphi_n\Vert _{L^\infty(Q)} \Vert \zw\Vert _{L^\infty(Q)}
  )
\end{align*}
The result follows gathering these estimates.

\hspace{-\parindent}(iii)  The result follows from (ii) noting that $\delta_n(w_a)-\delta_n(w_b)= [F'(\bar w)-F'(w_n)](w_a-w_b)$.

\hspace{-\parindent}(iv) In this case we use that
\[\delta_n(w) = F(\bar w)-F(w_n)-F'(w_n)(\bar w-w_n) +(F'(\bar w)-F'(w_n))(w-\bar w).\]
From (ii), we have that
\[\Vert [F'(\bar w)-F'(w_n)](w-\bar w)\Vert _{E_\infty}\leq L_{\bar w} \Vert \bar w- w_n\Vert _{W_p}
\Vert \bar w- w\Vert _{W_p}.\]
On the other hand, denoting $w_\theta = \theta\bar w+(1-\theta)w_n$, we obtain
\begin{align*}
 F(\bar w)-F(w_n)-F'(w_n)(\bar w-w_n)  = &  \int_0^1 [F'(w_\theta) -F'(w_n)](\bar w-w_n)\mathrm{d}\theta
\end{align*}
Using again (ii) and taking into account that $w_\theta-w_n =\theta(\bar w -w_n)$, we get
\[ \Vert (F'(w_\theta) -F'(w_n))(\bar w-w_n)\Vert _{E_\infty} \leq L_{\bar w} \Vert w_\theta-w_n\Vert _{W_p} \Vert \bar w-w_n\Vert _{W_p}\leq
L_{\bar w} \theta \Vert \bar w-w_n\Vert _{W_p}^2\]
The proof concludes using the identity $\int_0^1\theta\mathrm{d}\theta=\frac12$.
\end{proof}
For $\delta \in E_\infty$, consider $\zw^\delta\in\mY_A$ and $\etaw^\delta\in\mY_{A^*}$ solutions of
\begin{equation}\label{MME5.19}
\left\{\begin{array}{l}
  \partial_t\zw + A\zw + \partial_y a(\bar y)\zw =
  -\delta_1\ \text{ in }Q, \\
  \partial_{n_A}\zw + \bar u\zw = -\delta_2\ \text{ on }\Sigma,\ \zw(0) = -\delta_3\ \text{ in }\Omega,
\end{array}\right.
\end{equation}
\begin{equation}\label{MME5.20}
\left\{\begin{array}{l}
-\partial_t\etaw + A^*\etaw+\partial_y a(\bar y)\etaw = -\delta_4 + \left[\partial^2_{yy} L(\bar y) -\bar\varphi\partial^2_{yy} a(\bar y)\right] \zw^\delta\ \text{ in } Q,\\
\partial_{n_{A^*}}\etaw + \bar u\etaw = -\delta_5 \ \text{ on }\Sigma,\
\etaw(T) = - \delta_6 + \partial^2_{yy}\ l(\bar y(T)) \zw^\delta(T) \ \text{ in }\Omega.
\end{array}\right.
\end{equation}
The triplet $w_\delta=(y_\delta,\varphi_\delta,u_\delta)$ solves \eqref{eq::pertuebedproblem} if and only if $y_\delta = \bar y+\zw_\delta$, $\varphi_\delta = \bar\varphi+\etaw_\delta$ and $u_\delta = \bar u+v_\delta$ where $\zw_\delta = \zw_{\bar w,v_\delta} + \zw^\delta$, $\etaw_\delta = \etaw_{\bar w,v_\delta} + \etaw^\delta$, and $v_\delta$ satisfy
\begin{equation}\label{MME5.21}
 \int_\Sigma \big(
\tikhonov v_\delta - \bar y\etaw_\delta - \zw_\delta\bar \varphi + \tikhonov \bar u -\bar y\bar \varphi -\delta_7 \big)
(u - v_\delta)\dx\dt\geq 0\ \forall u \in\widehat{\uad}-\{\bar u\}.
\end{equation}
\begin{lemma}\label{le::applySSC}For every $\delta\in E_\infty$, the equation \eqref{eq::pertuebedproblem} has a unique solution $w_\delta = (y_\delta,\varphi_\delta,u_\delta) \in\widehat{W}_{ad}$. The function $v_\delta = u_\delta-\bar u$ satisfies
\begin{align}
  v_\delta(x,t) = \proj_{[\alpha-\bar u(x,t),\beta-\bar u(x,t)]} &
  \left\{\frac{1}{\tikhonov}\left[\bar y\etaw_{\bar w,v_\delta}+ \zw_{\bar w,v_\delta}\bar \varphi - \tikhonov \bar u +\bar y\bar \varphi\right.\right.\notag\\
& \left. \ + \delta_7 + \bar y\etaw^\delta + \zw^\delta \bar\varphi
  \right](x,t)\Big\}\label{eq::projectionformula}
\end{align}
for a.a. $(x,t)\in\Sigma\setminus\Sigma^\tau_{\bar u}$. Finally, $\bar w$ is the unique solution of \eqref{eq::pertuebedproblem} if $\delta=0_E$.
\end{lemma}
\begin{proof}
Noting that \eqref{MME5.21} can be written as
\begin{equation}\label{MME5.22}
 \int_\Sigma \Big[\!
\big(\tikhonov v_\delta - \bar y\etaw_{\bar w,v} - \zw_{\bar w,v}\bar \varphi\big)\! + \! \big(\tikhonov \bar u -\bar y\bar \varphi -\delta_7 -\bar y\etaw^\delta - \zw^\delta \bar\varphi\big)\Big](u - v_\delta)\! \dx\dt\geq 0
\end{equation}
for all $u\in\widehat{\uad}-\{\bar u\}$, we conclude that \eqref{eq::pertuebedproblem}  is the first-order necessary optimality condition of the constrained quadratic problem
\[
(\widehat{Q_\delta}) \quad \min_{v\in \widehat{\uad}-\{\bar u\}} \mathcal{Q}_\delta(v),
\]
where
\[
\mathcal{Q}_\delta(v) = \frac{1}{2}D^2_{(y,u)}\mL(\bar w)(\zw_{\bar w,v},v)^2
+ \int_\Sigma [\tikhonov \bar u -\bar y\bar\varphi - \delta_7 - \bar y\etaw^\delta  - \zw^\delta \bar\varphi] v\dx\dt.
\]
The strict convexity of $(\widehat{Q_\delta})$ is deduced in the same way as in Lemma \ref{L5.8} for the problem $(\widehat{Q_n})$ \tcr{(in this case it is enough to use the second order condition \eqref{eq::sscLa} instead of Theorem \ref{T4.1})}. Therefore, it has a unique solution $v_\delta$. This solution is the only function of $\widehat{\uad}-\{\bar u\}$ satisfying $\mathcal{Q}'_\delta(v_\delta)(u - v_\delta) \ge 0$ for all $u \in \widehat{\uad}-\{\bar u\}$. Computing the derivative $\mathcal{Q}'_\delta(v_\delta)$ as we did for $\mathcal{Q}'_n(v)$ in the proof of Lemma \ref{L5.6}, we find that the inequality $\mathcal{Q}'_\delta(v_\delta)(u - v_\delta) \ge 0$ is exactly the same inequality as \eqref{MME5.22}. Therefore, $w_\delta=(y_\delta,\varphi_\delta,u_\delta)$ is the unique solution to \eqref{eq::pertuebedproblem} with $y_\delta = \bar y+\zw_\delta$, $\varphi_\delta = \bar\varphi+\etaw_\delta$, and $u_\delta = \bar u+v_\delta$, where $\zw_\delta = \zw_{\bar w,v_\delta} + \zw^\delta$, $\etaw_\delta = \etaw_{\bar w,v_\delta} + \etaw^\delta$.

Identity \eqref{eq::projectionformula} follows from \eqref{MME5.22} with an standard argument. Finally, it is clear that $\bar w$ solves \eqref{eq::pertuebedproblem} when $\delta=0_E$. \end{proof}

\begin{theorem}\label{th::lip}(Lipschitz stability) Consider $\delta_a,\delta_b\in E_\infty$ and let $v_a,v_b\in L^\infty(\Sigma)$ be the solutions of $(\widehat{Q}_{\delta_a})$ and $(\widehat{Q}_{\delta_b})$ respectively. We set $\zw_a = \zw_{\bar w,v_a} + \zw^{\delta_a}$, $\etaw_a = \etaw_{\bar w,v_a} + \etaw^{\delta_a}$, and $w_a=(\zw_a,\etaw_a,v_a)$, where $\zw^{\delta_a}$ and $\etaw^{\delta_a}$ are the solutions of \eqref{MME5.19} and \eqref{MME5.20}, respectively, for $\delta = \delta_a$.  We define $w_b$ analogously. Then, there exists a constant $\hat c$ such that
\[
\Vert w_a - w_b\Vert _{W_q} \le \hat c\Vert \delta_a - \delta_b\Vert _{E_\infty}\ \text{ for all } q\in[ p,\infty].
\]
\end{theorem}

\begin{proof} Obviously it is enough to prove the inequality for $q = \infty$. First, we prove the existence of $\tilde c$ such that $\Vert v_a-v_b\Vert _{L^\infty(\Sigma)}\leq \tilde c\Vert \delta_a-\delta_b\Vert _{E_\infty}$. Denote $v=v_a-v_b$ and $\delta = \delta_a-\delta_b$. We start noting that from Theorem \ref{AT1} and the definition of $E_\infty$ we infer the existence of a constant $\bar c>0$ such that
\begin{equation}\label{MME5.28}
\Vert \delta_7 +\bar y\etaw^\delta  + \zw^\delta \bar\varphi\Vert _{L^r(\Sigma)}\leq \bar c \Vert \delta\Vert _{E_\infty}\quad \forall r \ge 2.
\end{equation}
\tcr{Thanks to \eqref{eq::sscLa},} we can use the stability result in $L^2(\Sigma)$ proved in \cite[Lemma A.1]{CM2025b} to  estimate
\[
\Vert v\Vert _{L^2(\Sigma)} \leq \frac{\tcr{1}}{\nu} \Vert \delta_7 +\bar y\etaw^\delta  + \zw^\delta \bar\varphi\Vert _{L^2(\Sigma)} \leq \frac{\bar c}{\nu}\Vert \delta\Vert _{E_\infty}.
\]
Now we use a bootstrapping argument. For $\dimension = 2$, we consider $p_0=2$, $p_1=3$, $3 < p_2 < \infty$, while for $\dimension = 3$,  $p_0=2$, $p_1 = 8/3$, $4 < p_2 < \frac{16}{3}$. In both cases $p_3=\infty$.  From Theorem \ref{AT5} we  have that for $i=1, 2, 3$ and for any $u\in L^{p_{i-1}}(\Sigma)$ there exists $c_i'>0$ and $c_i''>0$ such that
\begin{align}
  \Vert \zw_{\bar w,u}\Vert _{L^{p_i}(\Sigma)}\leq &  c'_i \Vert u\Vert _{L^{p_{i-1}}(\Sigma)},\label{MME5.26}\\
  \Vert \etaw_{\bar w,u}\Vert _{L^{p_i}(\Sigma)}\leq & c''_i \Vert u\Vert _{L^{p_{i-1}}(\Sigma)}.\label{MME5.27}
\end{align}
Let us also define
\[
\hat c_0 = \frac{\bar c}{\nu}\ \text{ and }\ \hat c_i = (c''_i \Vert \bar y\Vert _{L^\infty(Q)} + c'_i \Vert \bar \varphi\Vert _{L^\infty(Q)}) \frac{\hat c_{i-1}}{\tikhonov} + \bar c,\ i= 1, 2, 3.
\]
Subtracting the respective projection formulas for $v_b$ and $v_a$ detailed in \eqref{eq::projectionformula}, and using that the pointwise projection in $\mathbb{R}$ is a Lipschitz function with constant 1, we obtain for a.a. $(x,t)\in\Sigma\setminus\Sigma^\tau_{\bar u}$
\begin{align*}
  |v(x,t)| & \leq \frac{1}{\tikhonov}
  \Big\vert   \bar y(x,t) \etaw_{\zw_{\bar w,v},v}(x,t) +
  \bar\varphi(x,t)\zw_{\bar w,v}(x)
   \\
   & + \delta_7(x,t))
  +\bar\varphi(x,t)\zw^\delta(x,t)+ \bar y(x,t)\etaw^\delta(x,t)
  \Big\vert
\end{align*}
Since $v_a(x,t)=v_b(x,t)$ for a.e. $(x,t)\in\Sigma^\tau_{\bar u}$, we can write
\begin{align}\tikhonov\Vert v\Vert _{L^{p_i}(\Sigma)} = &\, \tikhonov\Vert v\Vert _{L^{p_i}(\Sigma\setminus \Sigma^\tau_{\bar u})} \notag\\
\leq &\,  \Vert  \bar y\etaw_{\bar w,v}
+ \zw_{\bar w,v}\bar \varphi+ \delta_7 + \bar y\etaw^\delta  + \zw^\delta \bar\varphi \Vert _{L^{p_i}(\Sigma\setminus \Sigma^\tau_{\bar u})}\notag\\
\leq &\,  \Vert  \bar y\etaw_{\bar w,v}
+ \zw_{\bar w,v}\bar \varphi+ \delta_7 + \bar y\etaw^\delta  + \zw^\delta \bar\varphi \Vert _{L^{p_i}(\Sigma)}\notag\\
\leq &\, \Vert \bar y\Vert _{L^\infty(Q)} \Vert \etaw_{\bar w,v}\Vert _{L^{p_i}(\Sigma)} +
\Vert \bar\varphi\Vert _{L^\infty(Q)} \Vert \zw_{\bar w,v}\Vert _{L^{p_i}(\Sigma)} + \bar c \Vert \delta\Vert _{E_\infty}\notag\\
\leq &\,  c''_i \Vert \bar y\Vert _{L^\infty(Q)} \Vert v\Vert _{L^{p_{i-1}}(\Sigma)}  + c'_i \Vert \bar \varphi\Vert _{L^\infty(Q)} \Vert v\Vert _{L^{p_{i-1}}(\Sigma)}  + \bar c\Vert \delta\Vert _{E_\infty}\notag\\
\leq &\, [(c''_i \Vert \bar y\Vert _{L^\infty(Q)} + c'_i \Vert \bar \varphi\Vert _{L^\infty(Q)}) \frac{\hat c_{i-1}}{\tikhonov} + \bar c ] \Vert \delta\Vert _{E_\infty} \notag
=   \hat c_i \Vert \delta\Vert _{E_\infty},\label{E5.31}
\end{align}
and $\Vert v_a-v_b\Vert _{L^\infty(\Sigma)}\leq \tilde c\Vert \delta_a-\delta_b\Vert _{E_\infty}$ follows for $\tilde c = \frac{\hat c_3}{\tikhonov}$.

The estimates for $\zw_a - \zw_b$ and $\etaw_a - \etaw_b$ follow from this one and Theorem \ref{T2.6} and Remark \ref{R2.9}.
\end{proof}

For $q\in [p,\infty]$, we set
\[
\inc_{p,q} = \left\{\begin{array}{ll}
\max\{1,\vert \Sigma\vert ^{\frac{q-p}{qp}}\} & \text{ if } q <\infty, \\
\max\{1, \vert \Sigma\vert ^{\frac{1}{p}}\}& \text{ if } q=\infty,
\end{array}
\right.\quad \rho_q^\star = \min\{\frac{\rho_{\textsc{ssc}}}{n_{p,q}},
\frac{1}{2 n_{p,q}^2 \hat c L_{\bar w}}\},
\]
where $\rho_{\textsc{ssc}}$, $L_{\bar w}$ and $\hat c$ are introduced respectively in Theorem \ref{T4.1}, Lemma \ref{L5.4} and Theorem \ref{th::lip}. We remark that for all $w\in W_q$ the inequality $\Vert w\Vert _{W_p} \le \inc_{p,q}\Vert w\Vert _{W_q}$ holds.

\begin{lemma}\label{le::quadconv}For $q\in[p,\infty]$, assume $w_0\in  B^q_{\rho}(\bar w)\tcr{\cap\mW}$ for $\rho\in (0, \rho^\star_q]$. Then, the sequence spanned by \eqref{E5.12} is well defined, $w_n\in \widehat{W_{\mathrm{ad}}}\cap B^q_\rho(\bar w)$ for all $n\geq 1$, $w_n\to w$ in $W_q$ and
\begin{equation}\label{eq::quadest}\Vert w_{n+1}-\bar w\Vert _{W_q}\leq n_{p,q}^2\hat c L_{\bar w} \Vert w_n-\bar w\Vert ^2_{W_q}\ \forall n\geq 0.
\end{equation}
\end{lemma}
\begin{proof}
We prove by induction that $w_n\in B^q_\rho(\bar w)$ for all $n\geq 0$. By assumption, the statement is true for $n=0$. Suppose it is true for $n$ and let us prove it for $n+1$.

  By Lemma \ref{L5.8} and using $n_{p,q}\rho\leq \rho_{\textsc{ssc}}$, we know that \eqref{E5.12} has a unique solution $\tcr{w}_{n+1}\in\widehat{W_{\mathrm{ad}}}$ and that it solves $(\widehat{Q}_\delta)$ for $\delta = \delta_n(w_{n+1})$. Let $w\in W_\infty$ be the unique solution of $(\widehat{Q}_\delta)$ for $\delta = \delta_n(\bar w)$. Using the Lipschitz stability Theorem \ref{th::lip}, Lemma \ref{L5.4}(iii), the inclusion $W_q\hookrightarrow W_p$, the induction hypotheses, and that $n_{p,q}^2 \hat c L_{\bar w} \rho\leq 1/2$ we infer
  \begin{align*}
    \Vert w_{n+1}-w\Vert _{W_q} \leq & \hat c\Vert \delta_n(w_{n+1})-\delta_n(\bar w)\Vert _{E_\infty}    \leq n_{p,q}^2 \hat c L_{\bar w} \Vert \bar w-w_n\Vert _{W_q} \Vert w_{n+1}-\bar w\Vert _{W_q}\\
 \leq & n_{p,q}^2 \hat c L_{\bar w} \rho \Vert w_{n+1}-\bar w\Vert _{W_q} \leq \frac{1}{2} \Vert w_{n+1}-\bar w\Vert _{W_q}.
  \end{align*}
  Using this inequality we obtain
  \begin{align*}
    \Vert w_{n+1}-\bar w\Vert _{W_q}\leq & \Vert w_{n+1}- w\Vert _{W_q}+\Vert w-\bar w\Vert _{W_q} 
    \leq   \frac12 \Vert w_{n+1}-\bar w\Vert _{W_q}+\Vert w-\bar w\Vert _{W_q},
  \end{align*}
which leads to $\Vert w_{n+1}-\bar w\Vert _{W_q}\leq 2\Vert w-\bar w\Vert _{W_q}$. Using that $v=0$  is the solution of $(\widehat{Q}_\delta)$ for $\delta=0_E$ (see Lemma \ref{le::applySSC}), the Lipschit stability Theorem \ref{th::lip}, and Lemma \ref{L5.4}(iv) we get
\begin{align}
  \Vert w_{n+1}-\bar w\Vert _{W_q}\leq & 2\Vert w-\bar w\Vert _{W_q}
  \leq 2\hat c\Vert \delta_n(\bar w)\Vert _{E_\infty} \notag \\
  \leq  &  \hat c L_{\bar w} \Vert w_n-\bar w\Vert _{W_p}^2\leq n_{p,q}^2 \hat c L_{\bar w} \Vert w_n-\bar w\Vert _{W_q}^2.\label{eq:quadconv}
\end{align}
By induction hypothesis $\Vert w_n-\bar w\Vert _{W_q}\leq \rho$, so the condition $n_{p,q}^2 \hat c L_{\bar w} \rho\leq 1/2$ yields
\begin{equation}\label{eq::mediodelta}\Vert w_{n+1}-\bar w\Vert _{W_q}\leq \frac12\rho,\end{equation}
and we have that $w_{n+1}\in B^q_{\rho}(\bar w)$, the sequence is well defined, and estimate \eqref{eq::quadest} follows from \eqref{eq:quadconv}. This estimate leads in a standard way to
\[\Vert w_{n+1}-\bar w\Vert _{W_q}\leq \frac{1}{2^{2^n}} \frac{1}{n_{p,q}^2\hat c L_{\bar w}},\]
so we have that $w_n\to\bar w$ in $W_q$ and the proof is complete.
\end{proof}


\begin{corollary}\label{C5.8}Let $q$ belong to $[p,\infty]$ and suppose $w_0\in  B^p_\rho(\bar w)$ for $\rho\in (0, \rho^\star_p]$ such that
\begin{equation}\label{eq::condrho}
\hat c L_{\bar w}\rho^2\leq \rho^\star_q.
\end{equation}
 Then, the sequence spanned by \eqref{E5.12} is well defined, \tcr{$w_n\in \widehat{W_{\mathrm{ad}}}\cap B^q_\rho(\bar w)$ for all $n\geq 1$, $w_n\to w$ in $W_q$} and
\begin{equation}\label{eq::quadestinf}\Vert w_{n+1}-\bar w\Vert _{W_q}\leq n_{p,q}^2\hat c L_{\bar w} \Vert w_n-\bar w\Vert ^2_{W_q}\ \forall n\geq 1.
\end{equation}
\end{corollary}
\begin{proof}From estimate \eqref{eq::mediodelta} we have that $\Vert w_1-\bar w\Vert _{W_p}\leq \rho/2$.
  Using that $\bar u$ is the solution of $(\widehat{Q}_\delta)$ for $\delta=0_E$, the Lipschitz stability result Theorem \ref{th::lip} and Lemma \ref{L5.4}(iv)
  \[\Vert w_1-\bar w\Vert _{W_q}\leq \hat c\Vert \delta_0(w_1)\Vert _{E_\infty} \leq \hat c L_{\bar w}\rho^2.\]
  Therefore $w_1\in B^q_{\rho_q}(\bar w)$ for $\rho_q =\hat c L_{\bar w}\rho^2$. Condition \eqref{eq::condrho} implies that $\rho_q\in(0, \rho^\star_q]$, so Lemma \ref{le::quadconv} yields quadratic convergence in $W_q$ for the shifted sequence $w_n'=w_{n+1}$ for $n\geq 0$.
\end{proof}

Finally, we prove that $ w_{n+1}$ is a solution of the Newton step for the original problem and it is the unique solution in a fixed neighborhood of $\bar w$.

\begin{lemma}\label{L5.7}
Consider $\rho\in (0,\rho_p^\star]$ satisfying
\begin{equation}\label{eq::condrho2}
3(M_{\bar u}+M_{\bar u,0})\rho\leq\tau,
\end{equation}
where $M_{\bar u}$ and $M_{\bar u,0}$ are introduced in equations \eqref{E2.11} and \eqref{E3.7}, and $\tau$ is given in \eqref{E3.14}.
Suppose that $w_n\in B^p_\rho(\bar w)\tcr{\cap \mW}$ and let $w_{n+1}\in \widehat{\wad}$ be the unique solution of the modified Newton step \eqref{E5.12}. Then, $w_{n+1}$ is the unique solution of the Newton step \eqref{E5.2} in $B^p
_\rho(\bar w)$ and $v=u_{n+1}-u_n$ is a strict local solution in $L^2(\Sigma)$ of the quadratic problem $(Q_n)$ introduced in Lemma \ref{L5.6}.
\end{lemma}
\begin{proof}
  Since $w_{n+1}=(y_{n+1},\varphi_{n+1},u_{n+1})$ is the solution of \eqref{E5.12}, then, the triplet $(\zw,\etaw,v)=w_{n+1}-w_n$ satisfies \eqref{E5.3}, \eqref{E5.4} and the variational inequality
\begin{equation}\label{eq::vi}
  \int_\Sigma (\tikhonov u_{n+1} -y_n\etaw - \zw\varphi_n  -y_n\varphi_n)(u-u_{n+1})\dx\dt\geq 0\quad \forall u \in \widehat{\uad}.
\end{equation}
Let us check that this inequality holds for all $u \in \uad$. For a fixed $u \in \uad$, we define $\hat u \in \widehat{\uad}$ as
\[
  \hat u (x,t)=\left\{\begin{array}{cl}
                                \bar u(x,t) & \text{ if }(x,t)\in \Sigma^\tau_{\bar u}, \\
                                u(x,t) & \text{ if }(x,t)\in \Sigma\setminus\Sigma^\tau_{\bar u}.
                              \end{array}\right.
\]
Then we have
\begin{align*}
      \int_\Sigma &  \big(
\tikhonov   u_{n+1} - y_n\eta - \zw\varphi_n -y_n\varphi_n \big)
(u -  u_{n+1})\dx\dt \\*
=  & \int_{\Sigma^\tau_{\bar u}} \big(
\tikhonov  u_{n+1} - y_n\etaw - \zw\varphi_n -y_n\varphi_n \big)
(u -  u_{n+1})\dx\dt\\*
& +  \int_{\Sigma\setminus \Sigma^\tau_{\bar u}}  \big(
\tikhonov u_{n+1} - y_n\etaw - \zw\varphi_n -y_n\varphi_n \big)
(\hat u - u_{n+1})\dx\dt = I + II.
   \end{align*}
Since $ u_{n+1}(x,t)=\bar u(x,t) = \hat u(x,t)$ for a.e. $(x,t)\in \Sigma^\tau_{\bar u}$, from \eqref{eq::vi} we deduce that
\[
II = \int_\Sigma  \big(\tikhonov u_{n+1} - y_n\etaw - \zw\varphi_n -y_n\varphi_n \big)
(\hat u - u_{n+1})\dx\dt\geq 0.
\]
To show that $I\geq 0$, we first notice that $u_{n+1} = \alpha$ on $\Sigma^{\tau^+}_{\bar u}$ and $u_{n+1} = \beta$ on $\Sigma^{\tau^-}_{\bar u}$. Therefore we can write
\begin{align*}
  I    =  & \int_{\Sigma^{\tau^+}_{\bar u}} \big(
\tikhonov \alpha - y_n\etaw - \zw\varphi_n -y_n\varphi_n \big)
(u - \alpha)\dx\dt\\
 & + \int_{\Sigma^{\tau^-}_{\bar u}} \big(
\tikhonov \beta - y_n\etaw - \zw\varphi_n -y_n\varphi_n \big)
(u - \beta)\dx\dt = I_\alpha+I_\beta.
\end{align*}
We write know
\begin{equation}\label{eq::aux01}
  y_n\etaw + \zw\varphi_n +y_n\varphi_n = y_n\etaw + \zw\varphi_n +y_n(\varphi_n-\bar\varphi) +(y_n-\bar y)\bar\varphi +\bar y\bar \varphi.
\end{equation}
By assumption and Lemma \ref{le::quadconv},  $\Vert y_n-\bar y\Vert _{L^\infty(\Sigma)}\leq \rho$, $\Vert \varphi_n-\bar\varphi\Vert _{L^\infty(\Sigma)}\leq \rho$, $\Vert \zw\Vert _{L^\infty(\Sigma)}\leq 2\rho$, and $\Vert \etaw\Vert _{L^\infty(\Sigma)}\leq 2\rho$. Therefore, using the condition \eqref{eq::condrho2} we obtain
\begin{equation}\label{eq::aux04}\Vert  y_n\etaw + \zw\varphi_n +y_n(\varphi_n-\bar\varphi) +(y_n-\bar y)\bar\varphi \Vert _{L^\infty(\Sigma)}\leq 3(M_{\bar u}+M_{\bar u,0})\rho\leq\tau.\end{equation}
Noticing that for a.a. $(x,t)\in\Sigma^{\tau^+}_{\bar u}$, $\tikhonov\alpha -\bar y(x,t)\bar\varphi(x,t) >\tau$, relations \eqref{eq::aux01} and \eqref{eq::aux04} imply that
\begin{equation}
  [y_n\etaw + \zw\varphi_n +y_n\varphi_n](x,t) < \tau +\tikhonov\alpha -\tau =   \tikhonov\alpha\ \text{ a.e. in }\Sigma^{\tau^+}_{\bar u}.\label{eq::aux02}
\end{equation}
This inequality together with $u \geq \alpha$, implies that the integrand in $I_\alpha$ is nonnegative, and so is $I_\alpha$.

We also have that $\tikhonov\beta -\bar y(x,t)\bar\varphi(x,t) <- \tau$ for a.a. $(x,t)\in\Sigma^{\tau^-}_{\bar u}$. Therefore, using again \eqref{eq::aux01} and \eqref{eq::aux04}, we get
\begin{equation}\label{eq::aux03}
[y_n\etaw + \zw\varphi_n +y_n\varphi_n](x,t) > -\tau +\tikhonov\beta +\tau =  \tikhonov\beta \ \text{ a.e. in }\Sigma^{\tau^-}_{\bar u}.
\end{equation}
This inequality together with $u \leq \beta$, implies that the integrand in $I_\beta$ is nonnegative, and so is $I_\beta$. Therefore $u_{n+1}$ is a solution of \eqref{E5.2}.

Let us prove the uniqueness. Suppose that there exists another solution $\hat w = (\hat y,\hat\varphi,\hat u) \in \wad \cap B^{p}_\rho(\bar w)$. Setting $\zw =\hat y-y_n$ and $\etaw = \hat \varphi-\varphi_n$ we get
\[
\hat u(x,t) = \proj_{[\alpha,\beta]}\left[\frac{1}{\tikhonov}\Big([y_n\etaw + \zw\varphi_n +y_n\varphi_n](x,t)\Big)\right]\quad \text{ a.e. in } \Sigma.
\]
 Once again we have that $\Vert \etaw\Vert _{L^\infty(\Sigma)}\leq 2\rho$ and $\Vert \zw\Vert _{L^\infty(\Sigma)}\leq 2\rho$, so estimates \eqref{eq::aux02} and \eqref{eq::aux03} hold, and therefore $\hat u(x,t)=\alpha$ for a.e. $(x,t)\in\Sigma^{\tau^+}_{\bar u}$ and $\hat u(x,t)=\beta$ for a.e. $(x,t)\in\Sigma^{\tau^-}_{\bar u}$. Thus, we conclude that $\hat u\in\widehat{\uad}$ and $\hat w\in\widehat{\wad}$. But we know from Lemma \ref{L5.8} that there exists a unique solution of \eqref{E5.12} in $\widehat{\wad}$, therefore $\hat w = w_{n+1}$.

Let us prove the last claim. With the notation of Lemma \ref{L5.6}, i.e., $\zw =\zw_{w_n,v}+\zw^n_0$ and $\etaw = \etaw_{w_n,v}+\etaw_0^n(x,t)$, the cone of critical directions for problem $(Q_n)$ at $v$ is
\begin{align*}
C^n_v =  \{u \in L^2(\Sigma):\ & u(x,t)\geq 0\text{ if }v(x,t)= \alpha-u_n(x,t),\\
&  u(x,t)\leq 0\text{ if }v(x,t)= \beta-u_n(x,t),\\
 & u(x,t) =0 \text{ if } [\tikhonov v - y_n\etaw - \zw\varphi_n + \tikhonov u_n -y_n\varphi_n](x,t)\neq 0\}.
\end{align*}
Let us prove that the second-order condition $\mathcal{Q}_n''(v)u^2 > 0$ for all $u \in C^n_v\setminus\{0\}$ holds. The functional $\mathcal{Q}_n$ satisfies assumptions (A1) and (A2) of \cite[Theorem 2.3]{Casas-Troltzsch2012}, and hence this is a sufficient condition for local optimality in the sense of $L^2(\Sigma)$. Since $\mathcal{Q}_n$ is a quadratic function, it is elementary to deduce that
\[
\mathcal{Q}_n''(v)u^2 = D^2_{(y,u)}\mathcal{L}(w_n)(\zw_{w_n,u},u)^2\quad \forall u \in L^2(\Sigma).
\]
Noting that $\rho\leq\rho_{\text{ssc}}$, we can apply Theorem \ref{T4.1} and we know that
\[
D^2_{(y,u)}\mathcal{L}(w_n)(\zw_{w_n,u},u)^2\geq \nu\Vert u\Vert ^2_{L^2(\Sigma)}\ \forall u \in E^\tau_{\bar u}.
\]
The claim will follow if we prove that $C^n_v\subset E^\tau_{\bar u}$. Consider $u \in C^n_v$ and $(x,t) \in\Sigma^{\tau}_{\bar u}$. We will prove that $u(x,t)=0$ using the third condition that defines the cone $C^n_v$. By definition of $y_{n+1}$ and $\varphi_{n+1}$, we have that $\zw =y_{n+1}-y_n$ and $\etaw = \varphi_{n+1}-\varphi_n$, so once again we have that $\Vert \etaw\Vert _{L^\infty(\Sigma)}\leq 2\rho$, $\Vert \zw\Vert _{L^\infty(\Sigma)}\leq 2\rho$ and estimates \eqref{eq::aux02} and \eqref{eq::aux03} hold. Then we get
\begin{align*}
 &[\tikhonov v - y_n\etaw - \zw\varphi_n + \tikhonov u_n - y_n\varphi_n](x,t) \\
= &\, \tikhonov u_{n + 1}(x,t) - [y_n\etaw + \zw\varphi_n + y_n\varphi_n](x,t) > \tikhonov\alpha-\tikhonov\alpha = 0 \ \text{ in } \in\Sigma^{\tau^+}_{\bar u}.
\end{align*}
Analogously we get
\begin{align*}
 &[\tikhonov v - y_n\etaw - \zw\varphi_n + \tikhonov u_n - y_n\varphi_n](x,t) \\
= & \, \tikhonov u_{n + 1}(x,t) - [y_n\etaw + \zw\varphi_n + y_n\varphi_n](x,t) < \tikhonov\beta-\tikhonov\beta = 0 \ \text{ in } \in\Sigma^{\tau^-}_{\bar u}.
\end{align*}
We have proved that $[\tikhonov v - y_n\etaw - \zw\varphi_n + \tikhonov u_n -y_n\varphi_n](x,t)\neq 0$ for almost all $(x,t)\in\Sigma^\tau_{\bar u}$, hence $u(x,t)=0$. The proof is complete.
\end{proof}

We conclude with the main results of this section.
\begin{theorem}\label{VV-T5.10} Let $\bar u\in\uad$ be a local solution of \Pb satisfying the strict complementarity condition given in Definition \ref{D3.12} and the second-order optimality condition $J''(\bar u)v^2 >0$ $\forall v \in C_{\bar u} \setminus\{0\}$. Let $\bar y,\bar\varphi\in Y$ be respectively the solutions of \eqref{E3.9} and \eqref{E3.10} and denote $\bar w = (\bar y,\bar\varphi,\bar u)$.
Let $\rho\in(0,\rho_p^\star]$ satisfy \eqref{eq::condrho} and \eqref{eq::condrho2} and let $w_0=(y_0,\varphi_0,u_0)\in\tcr{\mW}$ be an element of $B_\rho^p(\bar w)$. Then, for every $n\geq 0$ $(Q_n)$ has a unique local solution $v$ such that $w_{n+1} = (y_{n+1},\varphi_{n + 1},u_{n + 1}) \in B^p_\rho(\bar w)$, where $y_{n+1}=y_n + \zw_{w_n,v}+\zw_0^n$, $\varphi_{n+1}=\varphi + \etaw_{w_n,v} + \etaw_0^n$, and $u_{n+1}=u_n+v$. Moreover, the sequence $w_n$ converges to $\bar w$ in $W_\infty$ and we have
\begin{align}
&\Vert w_{n+1}-\bar w\Vert _{W_q}\leq n_{p,q}^2\hat c L_{\bar w} \Vert w_n-\bar w\Vert ^2_{W_q}\ \text{for all }n\geq 1\text{ and all }q\in[p,\infty].\label{MM-E5.32}
\end{align}
\end{theorem}

\begin{proof}
For $w_n\in B^p_\rho(\bar w)\tcr{\cap \mW}$, by Lemma \ref{L5.7} we know that problem $(Q_n)$ has a unique strict local solution $v$ such that the triplet $w_{n+1}$ built as stated in the theorem is the unique solution of \eqref{E5.2} in $B_\rho^p(\bar w)$. Therefore, the sequence is well defined. Also by Lemma \ref{le::quadconv} we know that $w_{n+1}\in\widehat{\wad}$ and solves \eqref{E5.12}. Consequently, the inequality \eqref{MM-E5.32} follows from Corollary \ref{C5.8}.
\end{proof}

\begin{corollary}\label{R5.12}Under the notation of Theorem \ref{VV-T5.10}
  there exists $r^\star>0$ such that if $w_0\in \wad\cap B^2_{r^\star}(\bar w)$ then the sequence $w_n$ converges quadratically to $\bar w$ in $W_q$ for all $q\in[p,\infty]$. 
\end{corollary}

\begin{proof}
  Let $\mu>0$ be defined as
  \[\mu = \vert \umax-\umin\vert ^{\frac{p-2}{p}}.\]
Since  $-\infty<\umin<\umax<\infty$, the following inequality holds:
  \[\Vert u-\bar u\Vert _{L^p(\Sigma)}\leq \mu\Vert u-\bar u\Vert _{L^2(\Sigma)}^{\frac{2}{p}}\ \forall u\in\uad.\]
   Take $\rho\in(0,\rho_p^\star]$ satisfying \eqref{eq::condrho2}  and $\hat c L_{\bar w}\rho^2\leq \rho^\star_\infty$. Since $\rho^\star_\infty \leq \rho^\star_q$, $\rho$ also satisfies \eqref{eq::condrho}.    
   Define now
   \[r^\star = \min\left\{\frac{\rho}{3},(\frac{\rho}{3\mu})^{\frac{p}{2}}\right\},\] 
If we take $w_0\in \wad\cap B^2_{r^\star}(\bar w)$, then $w_0\in B^p_\rho(\bar w)$ and the result follows from \eqref{MM-E5.32}.
\end{proof}

\begin{remark}
From the practical point of view, Corollary \ref{R5.12} tells us that we need not choose an initial control $u_0$ close to $\bar u$ in the sense of $L^\infty(\Sigma)$, but only in the sense of $L^2(\Sigma)$, which generally can be seen as easier from an intuitive point of view. Nevertheless, we are still left with the problem of choosing $y_0$ and $\varphi_0$ close to $\bar y$ and $\bar\varphi$, respectively, in $L^\infty(Q)$. As other authors have noticed, e.g. \cite[Theorem 6.15]{Hoppe-Neitzel2021}, one possibility is to solve the non-linear state equation for $y_{u_0}$ and the adjoint state equation for $\varphi_{u_0}$ and take these as initial $y_0$ and $\varphi_0$. Another possibility appropriate for the finite dimensional approximation \Pb, is discussed in Section \ref{S6}.
\end{remark}

\section{Computational considerations and a numerical example}\label{S6}
\setcounter{equation}{0}
Let us focus on an example problem. With the notation of Section \ref{S1}, we consider $\Omega=(0,1)^\dimension$ for $\dimension = 2,3$, and $T=4$. The data for the equation are $Ay = -\Delta y$, $a(x,t,y) = y^3-y$, $y_0(x) = \prod_{i=1}^{d}8 x_i(1-x_i)$, and $g(x,t) = 1$. To define the objective functional, we use $L(x,t,y) = \frac{1}{2}(y-y_d(x,t))^2$ with $y_d(x,t) = y_0(x)\cos\left(\pi t\right)$, $l(x,y)=0$, and  $\tikhonov = 0.3$. The control constraints are given by $\umin = 0.1$ and $\umax = 100$.

To discretize the problem we use a discontinuous Galerkin scheme in time, computationally equivalent to the implicit Euler method, and continuous piecewise linear finite elements in space for all the three variables.
The Tikhonov regularization term is discretized using mass lumping; see \cite{CasasMateosRosch2018} for the details of such technique.
Let us denote $\sigma=(h,\tau)$ the discretization parameters, with $h$ the spatial discretization parameter and $\tau$ the time discretization parameter. We name respectively $U_\sigma$, $Y_\sigma$ and $W_\sigma$ the approximation spaces of $L^p(\Sigma)$, $\mY_{A}$ (which in this case is equal to $\mY_{A^*}$) and $W$.
For some $i\geq 2$, we use a uniform mesh of size  $h_i=2^{-i}$ and time step $\tau_i=T\cdot 2^{-i}$. We get to the refinement $i=8$ if $\dimension =2$ and $i=5$ if $\dimension=3$. In this way, we have that in dimension $\dimension=2$, $\dim U_\sigma = 1\,024\times 256 = 262\,114$, $\dim Y_\sigma = 66\,049\times 256 = 16\,908\,544$ and hence $\dim W_\sigma = 34\,079\,232$. In dimension $\dimension=3$, $\dim U_\sigma=6\,146\times 32 = 196\,672$, $\dim Y_\sigma=  35\,937\times 32 = 1\,149\,984$, and hence $\dim W_\sigma = 2\,496\,640$.

Choosing the initial point is a delicate task. For $u_0$ close to $\bar u$ in $L^2(\Sigma)$, the choice $y_0=0$, $\varphi_0=0$ may fail to converge, because we may not be close enough  to $\bar w$ in the sense of $W_2$. In this case, $y_0=y_{u_0}$, $\varphi_0=\varphi_{u_0}$ will yield a point near the solution in the sense of $W_2$ and the algorithm will converge. This is for instance the behaviour of this example taking $u_0=0.6$, which is the average value of $\bar u$ that we had already computed by other means.

If $u_0$ is not close enough to $\bar u$ then the previous technique need not be successful. For instance in our problem $u_0=(\alpha+\beta)/2$, $y_0=y_{u_0}$, $\varphi_0=\varphi_{u_0}$ does not yield a convergent sequence.
To obtain a good initial point, we have used the following strategy: for $i\geq 3$,
the initial point to solve the problem at the refinement level $i$ is the output $(y,\varphi,u)$ obtained by the algorithm at the refinement level $i-1$.
 For the initial mesh at the level $i=2$, we solve the problem using another method that solves the nonlinear state equation at every iteration, taking as initial point $u_0=(\alpha+\beta)/2$. To implement this first step we have used the SQP method described in \cite{CM2025b}.

To stop the algorithm, we compute
\[\delta_\psi = \frac{\Vert \psi_{n}-\psi_{n-1}\Vert _{L^\infty(X)}}{\max\{1,\Vert \psi_n\Vert _{L^\infty(X)}\}}.\]
for each pair  $(\psi,X)\in\{(y,Q),(\varphi,Q),(u,\Sigma)\}$.
We stop when $\delta_u+\delta_y+\delta_\varphi < \rho$, with $\rho = 5\times 10^{-13}$ or when the last two objective values are equal up to machine precision.
We report on the processes in the final meshes in tables \ref{Table1} and \ref{Table2}.
The quadratic order of convergence can be clearly noted in the decrease of the order of magnitude of the errors.

 All the software has been programmed by us with \textsc{Matlab} and has been run on a desktop PC with 32GB of RAM and an Intel i7-13700F processor. In dimension $\dimension =2$, it takes $0.06$s to solve the problem for $i=2$, about $500$s to solve the five following problems for $i=3,\ldots,7$ in order to get the initial point for $i = 8$, and about $6500$s to solve the problem at the refinement level $i=8$, giving a total of about $7000$s. In dimension $\dimension =3$, it takes $0.21$s to solve the problem for $i=2$, about $75$s to solve the problems for $i=3,4$ in order to get the initial point for $i = 5$, and about $3225$s to solve the problem at the refinement level $i=5$, giving a total of about $3300$s.

\begin{table}
  \centering
  \begin{tabular}{ccccc}
    $n$ &       $\mJ(y_n,u_n)$            &   $\delta_u$   &   $\delta_y$   &   $\delta_\varphi$     \\ \hline
  0 &  8.0327338714522387e+00          \\
  1 &  8.0207916114504751e+00 &  5.5e-02  &  1.0e-01  &  1.6e-02  \\
  2 &  8.0208205341860719e+00 &  2.5e-03  &  2.4e-03  &  5.1e-04      \\
  3 &  8.0208203720226727e+00 &  1.0e-05  &  1.5e-06  &  7.0e-07   \\
  4 &  8.0208203720220883e+00 &  2.0e-12  &  1.4e-12  &  1.2e-12     \\
  5 &  8.0208203720220830e+00 &  2.2e-14  &  1.1e-14  &  9.1e-15
  \end{tabular}
  \caption{Convergence history of the problem in dimension $\dimension=2$; $h=2^{-8}$, $\tau = 2^{-8} T$.}\label{Table1}
\end{table}

\begin{table}
  \centering
  \begin{tabular}{ccccc}
    $n$ &       $\mJ(y_n,u_n)$            &   $\delta_u$   &   $\delta_y$   &   $\delta_\varphi$     \\ \hline
  0 &  1.3627846985203032e+01\\
  1 &  1.3429764090369611e+01 &  2.2e-01  &  6.7e-02  &  1.5e-01     \\
  2 &  1.3440976926011185e+01 &  2.4e-02  &  3.7e-03  &  1.1e-02   \\
 3 &  1.3441100627097502e+01 &  1.9e-04  &  1.7e-05  &  7.3e-05 \\
  4 &  1.3441100623224255e+01 &  7.0e-09  &  8.9e-10  &  2.7e-09   \\
  5 &  1.3441100623224251e+01 &  2.6e-15  &  7.3e-16  &  1.9e-15
  \end{tabular}
  \caption{Convergence history of the problem in dimension $\dimension=3$; $h=2^{-5}$, $\tau = 2^{-5} T$.}\label{Table2}
\end{table}

\appendix
\section{Some technical results}

\begin{theorem}
For every $u\in L^p(0,T;L^4(\Gamma))$ with $p>4$ and $u\geq 0$, $b\in L^\infty(Q)$, $f \in L^2(Q)$, $h \in L^2(\Sigma)$ and $y_0\in L^2(\Omega)$, the equation
\begin{equation}
\left\{\begin{array}{l} \displaystyle\frac{\partial y}{\partial t} + Ay + by = f\ \  \mbox{in } Q,\vspace{2mm}\\  \partial_{\conormal_A} y+ uy = h\ \ \mbox{on }\Sigma, \ y(x,0) = y_0(x) \ \ \text{in } \Omega \end{array}\right.
\label{AE1}
\end{equation}
has a unique solution $y\in W(0,T)$ and the following estimates hold:
\begin{align}
&\Vert y\Vert _Q \le L_Q\Big(\Vert f\Vert _{L^2(Q)} + \Vert h\Vert _{L^2(\Sigma)} + \Vert y_0\Vert _{L^2(\Omega)}\Big)\label{AE2}\\
&\Vert y\Vert _{W(0,T)}\leq L_W\Big(1 + \Vert u\Vert _{L^p(0,T;L^4(\Gamma))}\Big)\Big(\Vert f\Vert _{L^2(Q)} + \Vert h\Vert _{L^2(\Sigma)} + \Vert y_0\Vert _{L^2(\Omega)}\Big),\label{AE3}
\end{align}
where $L_Q$ and $L_W$ depend continuously and monotonically on $\Vert b\Vert _{L^\infty(Q)}$.
\label{AT1}
\end{theorem}

\begin{proof}
For every integer $k \ge 1$ we set $u_k(x,t) = \min\{u(x,t),k\}$. Since $u_k \in L^\infty(\Sigma)$, it is well known that the equation \eqref{AE1} with $u$ replaced by $u_k$ has a unique solution $y_k$ in $W(0,T)$. Moreover, testing the equation satisfied by $y_k$ with $y_k$ and using that $u_k \ge 0$ we infer the estimate \eqref{AE2} for $y_k$. Now, we prove the boundedness of $\{y_k\}_{k = 1}^\infty$ in $W(0,T)$. Testing the equation of $y_k$ with an arbitrary function $w \in H^1(\Omega)$ and using that $u_k \le u$ and the assumption \ref{A2.1} we get with H\"older inequality
\begin{align*}
\langle\frac{\partial y_k}{\partial t},w&\rangle_{H^1(\Omega)^*,H^1(\Omega)} \le C_A\Vert y_k(t)\Vert _{H^1(\Omega)}\Vert w\Vert _{H^1(\Omega)}\\
&+ \Big(\Vert f(t)\Vert _{L^2(\Omega)} + \Vert b(t)\Vert _{L^\infty(\Omega)}\Vert y_k(t)\Vert _{L^2(\Omega)}\Big)\Vert w\Vert _{L^2(\Omega)} \\
& + \Vert u(t)\Vert _{L^4(\Gamma)}\Vert y_k(t)\Vert _{L^2(\Gamma)}\Vert w\Vert _{L^4(\Gamma)} + \Vert h(t) \Vert_{L^{4/3}(\Gamma)} \Vert w(t)\Vert _{L^4(\Gamma)}\\
 \leq & C_1\Big(\Vert y_k(t)\Vert _{H^1(\Omega)} + \Vert f(t)\Vert _{L^2(\Omega)} + \Vert h(t) \Vert_{L^{4/3}(\Gamma)} \\
& + \Vert b(t)\Vert _{L^\infty(\Omega)}\Vert y_k(t)\Vert _{L^2(\Omega)}
+ \Vert u(t)\Vert _{L^4(\Gamma)}\Vert y_k(t)\Vert _{L^2(\Gamma)}\Big)\Vert w\Vert _{H^1(\Omega)}.
\end{align*}
This implies that
\begin{align*}
\Big\Vert \frac{\partial y_k}{\partial t}\Big\Vert _{L^2(0,T;H^1(\Omega)^*)} \le C_1\Big[&\Vert y_k\Vert _{L^2(Q)} + \Vert f\Vert _{L^2(Q)} +\Vert h \Vert_{L^{2}(\Sigma)} + \Vert b\Vert _{L^\infty(Q)}\Vert y_k\Vert _{L^2(Q)}\\
&+ \Big(\int_0^T\Vert u(t)\Vert ^2_{L^4(\Gamma)}\Vert y_k(t)\Vert ^2_{L^2(\Gamma)}\dt\Big)^{\frac{1}{2}}\Big].
\end{align*}
We have to estimate the last integral. Let us set $\theta = \frac{p - 2}{p}$. Since $p > 4$ we have that $\frac{1}{2} < \theta < 1$. Using H\"older inequality and that $\Vert z\Vert _{L^2(\Gamma)} \le C_\theta\Vert z\Vert _{H^\theta(\Omega)}$ for all $z \in H^\theta(\Omega)$, we get
\begin{align*}
\Big(\int_0^T\Vert u(t)\Vert ^2_{L^4(\Gamma)}\Vert y_k(t)\Vert ^2_{L^2(\Gamma)}\dt\Big)^{\frac{1}{2}} &\le C_\theta\Vert u\Vert _{L^p(0,T;L^4(\Gamma))}\Vert y_k\Vert _{L^{\frac{2}{\theta}}(0,T;H^\theta(\Omega))}\\
&\le C'_\theta\Vert u\Vert _{L^p(0,T;L^4(\Gamma))}\Vert y_k\Vert _Q.
\end{align*}
The last inequality follows by complex interpolation of the Banach spaces $L^\infty(0,T;L^2(\Omega))$ and $L^2(0,T;H^1(\Omega))$:
\[
[L^\infty(0,T;L^2(\Omega)),L^2(0,T;H^1(\Omega))]_\theta = L^{\frac{2}{\theta}}(0,T;[L^2(\Omega),H^1(\Omega)]_\theta) = L^{\frac{2}{\theta}}(0,T;H^\theta(\Omega));
\]
see \cite[\S 1.18.4]{Triebel1978}. From the established estimates we infer that every $y_k$ satisfies \eqref{AE3}. Therefore, we can take a subsequence, denoted in the same way such that $y_k \rightharpoonup y$ in $W(0,T)$ and $y_k \stackrel{*}{\rightharpoonup} y$ in $L^\infty(0,T;L^2(\Omega))$. Now it is immediate to check that $y$ is a solution of \eqref{AE1} and satisfies \eqref{AE2} and \eqref{AE3}. The uniqueness is a consequence of the estimate \eqref{AE2}.
\end{proof}

\begin{theorem}The trace operator is continuous
from $W(0,T)$ into $L^{2+2/d}(\Sigma)$ and compact from $W(0,T)$ into $L^{q}(\Sigma)$ for all $q<2+2/\dimension$.
\label{AT2}
\end{theorem}
\begin{proof}
Setting $\theta = \frac{\dimension}{\dimension + 1} \in (1/2,1)$ and using \cite[\S 1.18.4]{Triebel1978} we infer that $W(0,T) \subset L^\frac{2}{\theta}(0,T;H^\theta(\Omega))$. The trace operator is continuous from $H^\theta(\Omega)$ into $H^{\theta-1/2}(\Gamma)$ which is continuously embedded in $L^{2+2/d}(\Gamma)$. Since $2/\theta=2+2/d$, we finally obtain that the trace operator is continuous into $L^\frac{2}{\theta}(0,T;H^{\theta-1/2}(\Gamma))\subset L^{2+2/d}(\Sigma)$.

To prove the second part, we use \cite[Theorem 3]{Amann2001} with $X_0 = H^1(\Omega)^*$, $X_1 = H^1(\Omega)$, $p = 2$, $s = \frac{1}{2(\dimension + 1)}$, and deduce that $W(0,T)$ is compactly embedded in
\[
L^{\sigma}(0,T;(H^1(\Omega)^*,H^1(\Omega))_{\theta,1}) \subset  L^{\sigma}(0,T;(H^1(\Omega)^*,H^1(\Omega))_{\theta,2}) = L^{\sigma}(0,T;H^{2\theta - 1}(\Omega))
\]
with $\sigma = \frac{2}{1-2s} = 2 + \frac{2}{d}$ and for all $\theta \in (\frac{3}{4} , 1 - s)$.

The inequality $3/4<\theta$ implies $2\theta-1 > 1/2$, and hence the trace is continuous from $H^{2\theta - 1}(\Omega)$ into $H^{2\theta - 3/2}(\Gamma)$. The Sobolev embedding $H^{2\theta - 3/2}(\Gamma)\hookrightarrow L^q(\Gamma)$ holds for
\[
\frac{1}{q} = \frac{1}{2}-\frac{2\theta-3/2}{\dimension-1}.
\]
Taking into account that $\theta < 1-s$, and the value chosen for $s$, we deduce that
\[
q = \left\{\begin{array}{clc}
              \displaystyle \frac{1}{2(1-\theta)} & <\displaystyle\frac{1}{2s} = 3 = 2 + \frac{2}{d} & \text{ if }\dimension =2 \\
                \displaystyle\frac{4}{1+4(1-\theta)} & <\displaystyle\frac{4}{1+4s} = \frac{8}{3} = 2 + \frac{2}{d} & \text{ if }\dimension =3.
             \end{array}
\right.
\]
Since $\sigma = 2 + \frac{2}{d} > q$, we have that $L^{q^*}(0,T;L^q(\Gamma))\subset L^q(\Sigma)$, and the proof is complete.
\end{proof}

The above theorem implies the existence of a constant $C_\Sigma > 0$ such that
\begin{equation}
\Vert y\Vert _{L^q(\Sigma)}\leq C_\Sigma \Vert y\Vert _{W(0,T)}\quad \forall y \in W(0,T)\ \text{ and } \ \forall q \in \Big[2,2 + \frac{2}{\dimension}\Big].
\label{AE5}
\end{equation}
We also notice that the embedding $W(0,T)\hookrightarrow C([0,T];L^2(\Omega))$ implies the existence of a constant $C_\Omega>0$ such that
\begin{equation}\label{AE6a}\Vert y(T)\Vert_{L^2(\Omega)}\leq C_\Omega \Vert y\Vert _{W(0,T)}\quad \forall y \in W(0,T).
\end{equation}

\begin{theorem}\label{AT3}
\begin{enumerate}
\item  Assume that $p=2(\dimension+1)$, $f\in L^2(Q)$, $h\in L^2(\Sigma)$, and $y_0\in L^2(\Omega)$. Then, there exists an open set $\mA'$ of $L^p(\Sigma)$ such that $\mA_0\subset \mA'$ and for all $u\in\mA'$ the linear equation \eqref{AE1} has a unique solution $y\in W(0,T)$ and the following estimate holds
  \begin{equation}\label{AE6}
  \Vert y\Vert_{W(0,T)} \leq  L'_W \Big(\Vert f\Vert_{L^2(Q)}+\Vert h\Vert_{L^2(\Sigma)} +\Vert y_0\Vert_{L^2(\Omega)}\Big),
  \end{equation}
  where $L'_W$ depends continuously and monotonically on $\Vert u\Vert_{L^p(\Sigma)}$ and $\Vert b\Vert_{L^\infty(Q)}$.
  
\item If $p>d+1$, $f\in L^r(0,T;L^s(\Omega))$, $h\in L^{\hat r}(0,T;L^{\hat s}(\Gamma))$ and $y_0\in L^\infty(\Omega)$, then, there exists an open set $\mA''$ of $L^p(\Sigma)$ such that $\mA_0\subset \mA''$ and for all $u\in\mA''$ the linear equation \eqref{AE1} has a unique solution $y\in Y$ and it satisfies
  \begin{equation}\label{AE7}
  \Vert y\Vert_{L^\infty(Q)} \leq 2 K_\infty \Big(\Vert f\Vert_{L^r(0,T;L^s(\Omega))}+\Vert h\Vert_{L^{\hat r}(0,T;L^{\hat s}(\Gamma))} +\Vert y_0\Vert_{L^\infty(\Omega)}\Big),
  \end{equation}
  where $K_\infty$ is the constant given in Theorem \ref{T2.6}.
\end{enumerate}
\end{theorem}
\begin{proof}
  {\em 1. }
  Consider
  \[\mY_A^2 = \{y\in W(0,T): \frac{\partial y}{\partial t}+Ay+by\in L^2(Q)\text{ and }\partial_{\conormal_A} y\in L^2(\Sigma)\}.\]
  This is a Banach space when endowed with the graph norm.
  For $u \in L^p(\Sigma)$, define the continuous linear operator
  \begin{align*}
&\mT_{u} : \mY_{A}^2  \longrightarrow L^2(Q) \times L^2(\Sigma) \times L^2(\Omega)\\*
&\mT_{ u}y = \Big(\frac{\partial y}{\partial t} + Ay + by ,\partial_{\conormal_A}y +  u y,y(0) \Big).
\end{align*}
From Theorem \ref{AT1} we know that for every $u\in\mA_0$, the operator $\mT_{\bar u}$ is invertible. For $f\in L^2(Q)$, $h\in L^2(\Sigma)$, and $y_0\in L^2(\Omega)$, let $y\in W(0,T)$ be the unique solution of \eqref{AE1}. Then, using the definition of graph norm in $\mY_A^2$, the trace result Theorem \ref{AT2}, the fact that for $p=2(d+1)$ and $q=2+2/d$, the numbers $p/2$ and $q/2$ are Lebesgue conjugate exponents,  and estimate \eqref{AE3}, we get
\begin{align*}
  \Vert T_{\bar u}^{-1} &(f,h,y_0)\Vert_{\mY_A^2} =   \Vert y \Vert_{W(0,T)} + \Vert \frac{\partial y}{\partial t}+Ay+by\Vert_{L^2(Q)}
  + \Vert \partial_{\conormal_A} y\Vert_{L^2(\Sigma)}\\
  \leq & \Vert y \Vert_{W(0,T)} + \Vert f\Vert_{L^2(Q)} + \Vert h\Vert_{L^2(\Sigma)} + \Vert \bar u y\Vert_{L^2(\Sigma)}\\
  \leq & \Vert y \Vert_{W(0,T)} + \Vert f\Vert_{L^2(Q)} + \Vert h\Vert_{L^2(\Sigma)} + \Vert \bar u\Vert_{L^p(\Sigma)} \Vert y\Vert_{L^{2+2/d}(\Sigma)}\\
  \leq & \Vert y \Vert_{W(0,T)} + \Vert f\Vert_{L^2(Q)} + \Vert h\Vert_{L^2(\Sigma)} + C_\Sigma \Vert \bar u\Vert_{L^p(\Sigma)} \Vert y\Vert_{W(0,T)}\\
  \leq & \Big(1+C_\Sigma \Vert \bar u\Vert_{L^p(\Sigma)}\Big) L_W\Big(1 + \Vert \bar u\Vert _{L^p(\Sigma)}\Big)\Big(\Vert f\Vert _{L^2(Q)} + \Vert h\Vert _{L^2(\Sigma)} + \Vert y_0\Vert _{L^2(\Omega)}\Big)\\ &+ \Vert f\Vert_{L^2(Q)} + \Vert h\Vert_{L^2(\Sigma)}.   
\end{align*}
Taking
\[L_{\bar u} = 1+ L_W \Big(1+C_\Sigma \Vert \bar u\Vert_{L^p(\Sigma)}\Big) \Big(1 + \Vert \bar u\Vert _{L^p(\Sigma)}\Big)\]
we obtain that $\Vert \mT_{\bar u}^{-1}\Vert \leq L_{\bar u} $. Define $\varepsilon_{\bar u}'=\min\Big\{1,1/\big(2 C_\Sigma \Vert \mT_{\bar u}^{-1}\Vert\big)\Big\}$ and $\mA' = \cup_{u\in\mA_0} B_{\varepsilon_{\bar u}'}^p(\bar u)$. Let $u$ belong to $B_{\varepsilon_{\bar u}'}^p(\bar u)$. Then,
\begin{align*}
\Vert (\mT_{\bar u}- \mT_u)y \Vert = & \Vert (\bar u-u)y\Vert_{L^2(\Sigma)} \leq \Vert \bar u-u\Vert_{L^p(\Sigma)} \Vert y\Vert_{L^{2+\frac{2}{\dimension}}(\Sigma)} \\* 
\leq &C_\Sigma \Vert \bar u-u\Vert_{L^p(\Sigma)} \Vert y\Vert_{W(0,T)} \leq C_\Sigma\varepsilon_{\bar u}' \Vert y \Vert_{\mY_{A}^2}.
\end{align*}
Using that $\varepsilon_{\bar u}'\leq 1/\big(2 C_\Sigma \Vert \mT_{\bar u}^{-1}\Vert\big)$ we obtain $\Vert \mT_{\bar u}-\mT_u\Vert \leq 1/(2\Vert \mT_{\bar u}^{-1}\Vert)$, so $\mT_u$ is invertible, $\Vert \mT_u^{-1}\Vert \leq 2 \Vert \mT_{\bar u}^{-1}\Vert\leq 2L_{\bar u}$ and, using $\varepsilon_{\bar u}'\leq 1$, estimate \eqref{AE6} follows for
\[L'_W = 2+ 2L_W \Big(1+C_\Sigma \big(1+\Vert  u\Vert_{L^p(\Sigma)}\big)\Big) \Big(2 + \Vert  u\Vert _{L^p(\Sigma)}\Big).\]

{\em 2.} For every $\bar u\in\mA_0$, define $\varepsilon_{\bar u}'' = \min\big\{\varepsilon_{\bar u},1/(2K_\infty)\big\}$, where $\varepsilon_{\bar u}$ is introduced in the proof of Theorem \ref{T2.8} and $K_\infty$ is introduced in Theorem \ref{T2.6}. From Theorem \ref{T2.8} we deduce that, for $u\in B_{\varepsilon_{\bar u}''}^p(\bar u)$ equation \eqref{AE1} has a unique solution $y\in Y$. Writing \eqref{AE1} as
\[
\left\{\begin{array}{l} \displaystyle\frac{\partial y}{\partial t} + Ay + by = f\ \  \mbox{in } Q,\vspace{2mm}\\  \partial_{\conormal_A} y+ \bar u y = h - (u-\bar u)y\ \ \mbox{on }\Sigma, \ y(x,0) = y_0(x) \ \ \text{in } \Omega \end{array}\right.
\]
and applying \eqref{E2.3} we have the estimate
\[\Vert y\Vert_{L^\infty(Q)} \leq  K_\infty \Big(\Vert f\Vert_{L^r(0,T;L^s(\Omega))}+\Vert h\Vert_{L^{\hat r}(0,T;L^{\hat s}(\Gamma))} +\Vert y_0\Vert_{L^\infty(\Omega)} + \varepsilon_{\bar u}'' \Vert y \Vert_{L^\infty(Q)}\Big),\]
and the result follows noting that $\varepsilon_{\bar u}''\leq 1/(2K_\infty)$ and defining $\mA''=\cup_{\bar u\in\mA_0}B_{\varepsilon_{\bar u}''}^p(\bar u)$.
\end{proof}


\begin{remark}\label{AR4}
  Without loss of generality, we can assume that the $\varepsilon_{\bar u}$ found in the proof of Theorem \ref{T2.8} satisfies $\varepsilon_{\bar u}\leq \min\{\varepsilon_{\bar u}',1/(2K_\infty)\}$, so that $\varepsilon_{\bar u}'' = \varepsilon_{\bar u}$ in the previous result and $\mA\subset \mA'\cap\mA''$
\end{remark}

\begin{theorem}
For every $u\in L^p(\Sigma)$ with $p\tcr{=2(d+1)}$ 
 and $u\tcr{\in \mA}$, $b \in L^\infty(Q)$, and $v\in L^r(\Sigma)$, $2\leq r$, the unique solution $y\in W(0,T)$ of the equation
\[\left\{\begin{array}{l} \displaystyle\frac{\partial y}{\partial t} + Ay + by = 0 \mbox { in } Q,\vspace{2mm}\\  \partial_{\conormal_A}y+ uy = v \mbox{ on }\Sigma, \ y(x,0) = 0  \text{ in } \Omega. \end{array}\right.
\]
has a trace $y_{\vert\Sigma}\in L^{r^*}(\Sigma)$ for every $r^*$ satisfying
\[
\left\{\begin{array}{ll}\displaystyle r^* = 2 + \frac{2}{d} & \text{if } r = 2\vspace{2mm}\\\displaystyle r^* < \frac{(d^2-1)r}{(d+1-r)d} &\text{if } 2 < r < d+1,\\r^* < \infty &\text{if } r = d+1,\\ r^* = \infty& \text{if } r > d+1.\end{array}\right.
\]
Moreover, there exists a constant $c_{r^*}>0$ that depends continuously and monotonically on $\Vert b\Vert _{L^\infty(Q)}$ and $\Vert u\Vert _{L^p(\Sigma)}$ such that
\begin{equation}
\Vert y\Vert _{L^{r^*}(\Sigma)}\leq c_{r^*} \Vert v\Vert _{L^r(\Sigma)}.
\label{AE8}
\end{equation}
\label{AT5}
\end{theorem}

\begin{proof}
As a consequence of \tcr{estimate \eqref{AE6} and Theorem \ref{AT2}} we deduce that the theorem holds for $r = 2$. For $r > d + 1$ the $L^\infty(Q)$ regularity and the estimate \eqref{AE8} are 
\tcr{a consequence of estimate \eqref{AE7}.}
Let us consider the case $r = d+1$. For every $\varepsilon > 0$ we set $q_\varepsilon = d+1+\varepsilon$ and $\theta_\varepsilon = \frac{2\varepsilon}{(d+1)(d-1+\varepsilon)}$. Then, we have the interpolation relation
\[
\frac{1}{r} = \frac{\theta_\varepsilon}{2} + \frac{1 - \theta_\varepsilon}{q_\varepsilon}.
\]
Therefore, setting
\[
\frac{1}{r_\varepsilon^*} = \frac{\theta_\varepsilon}{2 + \frac{2}{d}} + \frac{1 - \theta_\varepsilon}{\infty} = \frac{\varepsilon d}{(d+1)^2(d-1+\varepsilon)},
\]
we get by interpolation that $y \in L^{r_\varepsilon^*}(\Sigma)$ and \eqref{AE8} holds. Since $r_\varepsilon^* \to \infty$ as $\varepsilon \to 0$, the result follows for $r = d + 1$.  Finally, we consider the case $2 < r < d+1$. In this case, for $\theta = \frac{2(d + 1 -r)}{(d - 1)r}$ we have
\[
\frac{1}{r} = \frac{\theta}{2} + \frac{1 - \theta}{d + 1}.
\]

Then, using again interpolation and taking into account that for $r = d+ 1$ we have $L^q(\Sigma)$ regularity for all $q < \infty$ we infer that $y_{\vert\Sigma}\in L^{r_q*}(\Sigma)$ for
\[
\frac{1}{r_q^*} = \frac{\theta}{2+ \frac{2}{d}} + \frac{1 - \theta}{q} =\frac{(d+1-r)d}{(d^2-1)r} + \frac{1 - \theta}{q}.
\]
Hence, for every $r^* < \frac{(d^2-1)r}{(d+1-r)d}$ we can take $q$ sufficiently large so that $r_q^* \ge r^*$ and consequently $y_{\vert\Sigma}\in L^{r*}(\Sigma)$ and \eqref{AE8} holds.
\end{proof}

\bibliographystyle{abbrvurl}
\bibliography{CCM}

\end{document}